\documentclass[11pt]{article}

\usepackage{amsmath}
\usepackage{amssymb}
\usepackage{amsthm}
\usepackage{titlesec}
\usepackage{graphicx}
\usepackage{caption}
\usepackage{subcaption}

\usepackage{diagbox}

\usepackage[T1]{fontenc}
\usepackage[utf8]{inputenc}

\usepackage{comment}

\usepackage{indentfirst}
\usepackage[top=1.25in,left=1.25in,right=1.25in]{geometry}
\newlength{\numone}
\newlength{\widone}
\newlength{\numtwo}
\newlength{\widtwo}
\newcounter{countp}

\newtheorem{thm}{Theorem}%[section]

\newtheorem{lemma}[thm]{Lemma}

\theoremstyle{definition}
\newtheorem{defin}[thm]{Definition}

\newtheorem{rem}[thm]{Remark}

\numberwithin{equation}{section}

\author{\Large{Riccardo W. Maffucci}}
\newcommand{\Addresses}{{
		\footnotesize
		
		R.W.~Maffucci, \textsc{EPFL MA SB,
			Lausanne, Switzerland 1015}\par\nopagebreak\vspace{-0.35cm}
		\textit{E-mail address}, R.W.~Maffucci: \texttt{riccardo.maffucci@epfl.ch}}}

\title{\Large{\uppercase{\bf On unigraphic $3$-polytopes of radius one}}}
\date{}

\def\calC{\mathcal{C}}

\def\calS{\mathcal{S}}

\newcommand{\C}{\mathbb{C}}

\begin{document}
\titleformat{\section}
  {\Large\scshape}{\thesection}{1em}{}
\titleformat{\subsection}
  {\large\scshape}{\thesubsection}{1em}{}
\maketitle

%\pagenumbering{roman}
%\addcontentsline{toc}{section}{Table of contents}
%\tableofcontents

\begin{abstract}
We ask which degree sequences admit a unique realisation as a $3$-polytopal graph (polyhedron) on $p$ vertices. We give an exhaustive list of these sequences for the case where one degree equals $p-1$ and exactly two or three of them equal $3$. We also find all $3$-polytopes of radius one with $p\leq 17$, and those with $q\leq 41$ edges, by developing a fast algorithm and making use of High Performance Computing.
\end{abstract}
{\bf Keywords:} Unigraphic, Unique realisation, Degree sequence, Valency, Planar graph, Graph radius, $3$-polytope, Enumeration, Graph algorithm.
\\
{\bf MSC(2010):} 05C07, 05C75, 05C62, 05C10, 52B05, 52B10, 05C30, 05C85.

%\tableofcontents

\section{Introduction}
\subsection{Results}
The distance $d(u,v)$ between two vertices $u,v$ of a (finite, simple) connected graph $F=(V,E)$ is the length of a minimal $uv$-path. The graph radius of $F$ is defined as
\[\text{rad}(F):=\min\{v\in V : \text{ecc}(v)\},\]
where
\[\text{ecc}(v):=\max\{w\in V : d(v,w)\}\]
is the eccentricity of $v$. In this paper we will refer to the graph radius simply as the radius.

The planar, $3$-connected graphs are the $1$-skeletons of $3$-polytopes \cite{radste}. In the literature these graphs are in fact called $3$-polytopes, or sometimes polyhedra. They are the planar graphs that are \textit{uniquely} embeddable in a sphere \cite{whit32}. Their regions are bounded by cycles (polygons), and this fact is true more generally for planar, $2$-connected graphs \cite[Proposition 4.26]{dieste}. For recent work on distance, radius, and related topics for graphs, see e.g. \cite{mukw12,aknisi,liwe13}. For distance topics in polytopes, see e.g. \cite{kkmnns,sant12}.

According to Tutte's Theorem \cite{tutt61}, if $F$ is a $3$-polytope of size (i.e. number of edges) $q$ that is not a pyramid (i.e. a wheel graph), then either $F$ or its dual may be obtained by adding an edge to a $3$-polytope of size $q-1$. This yields an algorithm to generate all $3$-polytopes. For the sub-family of $3$-polytopes of graph radius $1$, there is a faster algorithm to generate them all.

\begin{rem}
	\label{rem:1}
Every $3$-polytope of graph radius $1$ and size $q$ that is not a pyramid may be obtained by adding an edge to a $3$-polytope of radius $1$ and size $q-1$. To see this, observe that a graph $F$ is $3$-polytopal of radius $1$ if and only if $F-v$ is planar and has a region containing all remaining vertices, where $v$ is a vertex of eccentricity $1$ in $F$. If $F$ is not a pyramid, then there is an edge $e$ such that $F-v-e$ still has a region containing all remaining vertices. The author wishes to thank Lionel Br\"{u}tsch and Niels Willems for helping to make this point explicit.
\end{rem}

The degree of a vertex is the number of edges incident to it (vertices adjacent to it). Letting $V=\{v_1,v_2,\dots,v_p\}$ and $d_i=\deg(v_i)$ for $i=1,2,\dots,p$, we call
\begin{equation}
\label{eqn:s}
s: d_1,d_2,\dots,d_p
\end{equation}
the degree sequence of $F$. Vice versa, a sequence $s$ \eqref{eqn:s} is called graphic (sometimes `graphical') if there exists an order (i.e. number of vertices) $p$ graph $F$ of degree sequence $s$. We then say that $F$ is a realisation of $s$. The classical theorems of Havel \cite{have55} and Hakimi \cite{hakimi} and Erdös-Gallai \cite{erdgal} determine when $s$ is graphic. With the Havel--Hakimi algorithm \cite{have55,hakimi}, one can construct such an $F$ realising $s$. For recent work on degree sequences of graphs, see e.g. \cite{barr22,mukw12,liwe13}. For recent results on degree sequences of $3$-polytopes, see e.g. \cite{mafpo2,mafpo3}.

A sequence $s$ is called \textit{unigraphic} if, up to isomorphism, there is exactly one graph of degree sequence $s$. With some abuse of terminology, we will refer to the corresponding realisation $F$ as unigraphic. Koren \cite{koren1} and Li \cite{li1975} devised criteria to establish when a given $s$ is unigraphic. Unigraphic non-$2$-connected graphs have been classified \cite{john80}, and also those satisfying $d_2=d_{p-1}$ \cite[Theorem 6.1]{koren1}.%, there is no exhaustive list of families of unigraphic sequences, or an explicit, straightforward way to characterise them.

The problem might be more treatable if we ask which $s$ are uniquely realisable by a graph $F$ belonging to a given sub-class of graphs, i.e., satisfying certain properties. For instance, it is easy to explicitly list the types of unigraphic sequences of tree graphs \cite[Exercise 6.11]{harary}. Then it may or may not be the case that $s$ is unigraphic w.r.t. the class of all graphs, e.g. $2,2,2,1,1$ is unigraphic \textit{w.r.t. the sub-class of trees}, but not unigraphic w.r.t. the class of all graphs.

In this paper, we ask what are the unigraphic $3$-polytopal sequences, and attempt to answer this question for the class of $3$-polytopes of radius $1$. Plainly $\text{rad}(F)=1$ means that there exists a vertex, $v_1$ say, adjacent to all others, i.e. $d_1=p-1$. In what follows, we will denote by $a$ the number of degree $3$ vertices in $F$,
\begin{equation}
\label{eqn:a}
a=a(F):=\#\{v\in V : \deg(v)=3\}
\end{equation}
(this is the smallest possible degree, as $F$ is $3$-connected). It is not difficult to show that $a\geq 2$ (see Lemma \ref{le:SG} below). Our first result is a precise characterisation of the unigraphic, $3$-polytopal, radius $1$ sequences with $a=2$. The notation $d^k$ in a sequence stands for $k$ vertices of degree $d$.

\begin{thm}
	\label{prop:1}
	%A weakly-decreasing sequence $s$ \eqref{eqn:s} satisfying $d_1=p-1$, $d_{p-1}=3$, and $d_{p-2}>3$ is unigraphic w.r.t. the class of $3$-polytopal graphs if and only if $s$ is one of the following:
	Let $s$ be a sequence as in \eqref{eqn:s}, with exactly $2$ vertices of degree three. Then $s$ is unigraphic as a $3$-polytope of radius $1$ if and only if $s$ is one of the following:
	\begin{align*}
	A1&: p-1, 4^{p-3}, 3^2, & & p\geq 5, \ p \text{ odd};
	\\A2&: p-1, p-3, 4^{p-4}, 3^2, & & p\geq 8;
	\\A3&: p-1,(x+3)^{(p-5)/(x-1)},4^{(p-5)(x-2)/(x-1)+2},3^2, & & p\geq 6, \ (p-5)/(x-1)\in\mathbb{N};
	\\A4&: p-1, x+3, p-x, 4^{p-5}, 3^2, & & p\geq 8, \ \lfloor(p-1)/2\rfloor\leq x\leq p-5.
	\end{align*}
\end{thm}

%\begin{defin}
%	\label{def:1}

%\end{defin}

\begin{comment}
\begin{thm}
	\label{prop:1}
Let $F$ be a $3$-polytope of order $p$ and graph radius $1$, and $s$ its degree sequence. Assume that the number of degree three vertices is $2$ or $3$. Then $G$ is unigraphic if and only if $s$ is one of the following:
\begin{align*}
A1&: p-1, 4^{p-3}, 3^2, & & p\geq 5, \ p \text{ odd};
\\A2&: p-1, p-3, 4^{p-4}, 3^2, & & p\geq 8;
\\A3&: p-1,(x+3)^{(p-5)/(x-1)},4^{(p-5)(x-2)/(x-1)+2},3^2, & & p\geq 6, \ (p-5)/(x-1)\in\mathbb{N};
\\A4&: p-1, x+3, p-x, 4^{p-5}, 3^2, & & p\geq 8, \ \lfloor(p-1)/2\rfloor\leq x\leq p-5;
\\& \hspace{0.3cm} 5,4,4,3,3,3;
\\& \hspace{0.3cm} 6,5,5,5,3,3,3;
\\B1&: p-1, 6, 5^{p-6}, 4, 3^3, & & p\geq 8;
\\B2&: p-1, (x+3)^2, p-1-2x+3, 4^{p-7}, 3^3, & & p\geq 10, \ 3\leq x\leq\lfloor(p-4)/2\rfloor;
\\B3&: p-1, \left(\frac{p}{4}+3\right)^4, 4^{p-8}, 3^3, & & p\geq 12, \ p\equiv 0\hspace{-0.25cm}\pmod 4.
\end{align*}
\end{thm}
\end{comment}
Theorem \ref{prop:1} will be proven in section \ref{sec:2}. With similar methods but an increasing amount of work, one can go about proving a statement of this flavour for the cases $a=3$, $a=4$, $\dots$. Here we prove the analogue of Theorem \ref{prop:1} for $a=3$, using similar methods. Formulating a general statement for any feasible $a$ seems to be difficult. A result in this direction will appear elsewhere \cite{mprep00}.

\begin{thm}
	\label{prop:2}
	%Let $F$ be a $3$-polytope of order $p$ and graph radius $1$, and $s$ its degree sequence. Assume that the number of degree three vertices is $3$. Then $s$ is unigraphic if and only if it is one of the following:
	Let $s$ be a sequence as in \eqref{eqn:s}, with exactly $3$ vertices of degree three. Then $s$ is unigraphic as a $3$-polytope of radius $1$ if and only if $s$ is one of the following:
	\begin{align*}
	\\B1&: p-1, (x+3)^2, p-1-2x+3, 4^{p-7}, 3^3, & & p\geq 10, \ 3\leq x\leq\lfloor(p-4)/2\rfloor;
	\\C1&: p-1, 6, 5^{p-6}, 4, 3^3, & & p\geq 8;
	\\D1&: p-1, \left(\frac{p}{4}+3\right)^4, 4^{p-8}, 3^3, & & p\geq 12, \ p\equiv 0\hspace{-0.25cm}\pmod 4,
	\end{align*}
	or one of the exceptional $5,4^2,3^3$ and $6,5^3,3^3$.
\end{thm}
Theorem \ref{prop:2} will be proven in section \ref{sec:3}.

To describe the $3$-polytopes corresponding to these sequences, we need some extra notation. For $F$ a $3$-polytope of graph radius one satisfying $a=2$ or $a=3$, $F-v_1$ is a non-empty, planar, Hamiltonian graph. It also has a region containing all vertices (except $v_1$ of course).
\begin{defin}
	\label{def:G}
	In what follows, $H$ will denote a Hamiltonian cycle in $F-v_1$, and $E(H)$ its edge set. If $a=2$ or $a=3$, clearly $F$ is not a pyramid, and we define $G:=F-v_1-E(H)$, a non-empty planar graph, of sequence
	\begin{equation*}
	s': d_2-3, d_3-3, \dots, d_{p}-3.
	\end{equation*}
\end{defin}
Note that the unigraphicity of $s$ as a $3$-polytope of radius $1$ does not imply that of $s'$ among planar graphs, and viceversa. The notation $H,G,s'$ will always be as in Definition \ref{def:G}.

In all of the types A1, A2, A3, and A4, $G$ is a forest, and either it has at most two non-trivial trees, or all non-trivial trees are copies of $K_2$ (type A1). The types B1, C1, and D1 have in common that $G$ has only one non-trivial component, and exactly one cycle, that is a triangle. For B1, $G$ is formed (apart from isolated vertices) of a triangle together with vertices of degree one adjacent to the boundary points of the triangle. At least two of these points have the same degree in $G$. For C1, the only non-trivial component of $G$ is a triangle together with a path starting from one of its vertices (see e.g. Figures \ref{pic:011} and \ref{pic:012}). For D1, the non-trivial component of $G$ is a triangle together with a fourth vertex adjacent to exactly one point on the boundary of the triangle, and with extra degree one vertices adjacent each to one of these four vertices, so that these four have the same degree in $G$ (see e.g. Figure \ref{pic:022}).

Moreover, we have gathered the data of Tables \ref{tab:1}, \ref{tab:2}, and \ref{tab:3} by implementing the simplified faster version of Tutte's algorithm for the case of radius one (Remark \ref{rem:1}), with the help of Scientific IT \& Application Support
(SCITAS) High Performance Computing (HPC) for the EPFL community. In particular, we found all $3$-polytopes of radius $1$ with $p\leq 17$, and also those with $q\leq 41$, and among these the unigraphic ones. The unigraphic $3$-polytopes found with this algorithm are consistent with the results of Theorems \ref{prop:1} and \ref{prop:2}.

\begin{table}[h!]
	\centering
	%\begin{equation}
	$\begin{array}{|c|c|c|c|c|c|c|c|}
	\hline
	\text{edges}&3\text{-polytopes}&3\text{-polytopes of radius }1&\text{rad. }1\text{ sequences}&\text{unigraphic rad. }1\text{ seq.}\\
	\hline\hline
	6&1&1&1&1\\\hline
	7&0&0&0&0\\\hline
	8&1&1&1&1\\\hline
	9&2&1&1&1\\\hline
	10&2&1&1&1\\\hline
	11&4&1&1&1\\\hline
	12&12&2&2&2\\\hline
	13&22&2&1&0\\\hline
	14&58&4&3&2\\\hline
	15&158&5&4&3\\\hline
	16&448&7&3&1\\\hline
	17&1342&10&5&2\\\hline
	18&4199&16&7&5\\\hline
	19&13384&27&6&1\\\hline
	20&43708&42&10&3\\\hline
	21&144810&67&15&6\\\hline
	22&485704&116&11&2\\\hline
	23&1645576&187&18&2\\\hline
	24&5623571&329&28&11\\\hline
	25&19358410&570&21&1\\\hline
	26&67078828&970&35&4\\\hline
	27&?&1723&52&12\\\hline
	28&&3021&38&1\\\hline
	29&&5338&61&3\\\hline
	30&&9563&90&15\\\hline
	31&&16981&67&1\\\hline
	32&&30517&103&3\\\hline
	33&&54913&158&18\\\hline
	34&&98847&112&2\\\hline
	35&&179119&178&3\\\hline
	36&&324333&258&20\\\hline
	37&&589059&191&1\\\hline
	38&&1072997&287&3\\\hline
	39&&1955207&425&24\\\hline
	40&&3573129&?&1\\\hline
	41&&6538088&?&1\\\hline
	\end{array}$
	%\end{equation}
	\caption{Number of $3$-polytopes or radius $1$, degree sequences and unigraphic sequences up to $41$ edges. For the total number of $3$-polytopes on $q$ edges see e.g. Dillencourt \cite{dillen}.}
	\label{tab:1}
\end{table}

\begin{table}[h!]
	\footnotesize
	\centering
	%\begin{equation}
	$\begin{array}{|c||c|c|c|c|c|c|c|c|c|c|c|c|c|c|c|c|c|c|}
	%\hline
	%\text{edges}&3\text{-polytopes}&3\text{-polytopes of radius }1&\text{radius }1\text{ sequences}&\text{unigraphic radius }1\text{ sequences}\\
	\hline
	\hbox{\diagbox{$q$}{$p$}}&4&5&6&7&8&9&10&11&12&13&14\\
	\hline\hline
	6&1(1)\\\hline
	7&\\\hline
	8&&1(1)\\\hline
	9&&1(1)\\\hline
	10&&&1(1)\\\hline
	11&&&1(1)\\\hline
	12&&&1(1)&1(1)\\\hline
	13&&&&2&0\\\hline
	14&&&&3(1)&1(1)\\\hline
	15&&&&3(3)&2\\\hline
	16&&&&&6&1(1)\\\hline
	17&&&&&7(2)&3\\\hline
	18&&&&&4(4)&11&1(1)\\\hline
	19&&&&&&24(1)&3\\\hline
	20&&&&&&24(2)&17&1(1)\\\hline
	21&&&&&&12(6)&51&4\\\hline
	22&&&&&&&89(1)&26&1(1)\\\hline
	23&&&&&&&74(2)&109&4\\\hline
	24&&&&&&&27(9)&265(1)&36&1(1)\\\hline
	25&&&&&&&&371(1)&194&5\\\hline
	26&&&&&&&&259(3)&660&50&1(1)\\\hline
	27&&&&&&&&82(11)&1291(1)&345&5\\\hline
	28&&&&&&&&&1478&1477&65\\\hline
	29&&&&&&&&&891(2)&3891(1)&550\\\hline
	30&&&&&&&&&228(14)&6249&3000\\\hline
	31&&&&&&&&&&6044(1)&10061\\\hline
	32&&&&&&&&&&3176(2)&21524\\\hline
	33&&&&&&&&&&733(18)&29133\\\hline
	34&&&&&&&&&&&24302\\\hline
	35&&&&&&&&&&&11326(3)\\\hline
	36&&&&&&&&&&&2282(19)\\\hline
	\end{array}$
	%\end{equation}
	\caption{Number of $3$-polytopes or radius $1$, sorted by order $p$ and size $q$ (continues in Table \ref{tab:3}). Numbers in brackets indicate unigraphic $3$-polytopes.}
	\label{tab:2}
\end{table}

\begin{table}[h!]
	\footnotesize
	\centering
	%\begin{equation}
	$\begin{array}{|c||c|c|c|c|c|c|c|c|c|c|c|c|c|c|c|c|c|}
	%\hline
	%\text{edges}&3\text{-polytopes}&3\text{-polytopes of radius }1&\text{radius }1\text{ sequences}&\text{unigraphic radius }1\text{ sequences}\\
	\hline
	\hbox{\diagbox{$q$}{$p$}}&15&16&17&18&19&20&21\\
	\hline\hline
	28&1(1)\\\hline
	29&6\\\hline
	30&85&1(1)\\\hline
	31&870&6\\\hline
	32&5710&106&1(1)\\\hline
	33&23747&1293&7\\\hline
	34&64183(1)&10228&133&1(1)\\\hline
	35&114541&51349&1896&7\\\hline
	36&133464&170904&17521&161&1(1)\\\hline
	37&98000(1)&384035&104349&2667&8\\\hline
	38&40942(2)&586696&416385&28777&196&1(1)\\\hline
	39&7528(23)&599516&1144304(1)&200137&3714&8\\\hline
	40&&392528&2192206&942417&45745&232&1(1)\\\hline
	41&&148646(1)&2923018&3094421&366982&5012&9\\\hline
	42&&24834(23)&2656742&\dots&\dots&\dots&\dots\\\hline
	43&&&1570490&\dots&\dots&\dots&\dots\\\hline
	44&&&543515(2)&\dots&\dots&\dots&\dots\\\hline
	45&&&83898(28)&\dots&\dots&\dots&\dots\\\hline
	\end{array}$
	%\end{equation}
	\caption{Number of $3$-polytopes or radius $1$, sorted by order $p$ and size $q$ (continued from Table \ref{tab:2}). Numbers in brackets indicate unigraphic $3$-polytopes.}
	\label{tab:3}
\end{table}

\clearpage
\paragraph{Plan of the rest of the paper.}
We end the introduction with the general setup and initial considerations. In section \ref{sec:2} we will prove Theorem \ref{prop:1}, i.e. the case $a=2$, and in section \ref{sec:3} we will prove Theorem \ref{prop:2}, i.e. the case $a=3$.

%\begin{prop}
%	\label{prop:2}
%	Let $G$ be a $(p,q)$ $3$-polytope of graph radius $1$ that is not a pyramid, and $s$ be its degree sequence. Assume that the number of degree three vertices is $d\geq 4$. If $G''$ contains a caterpillar of length $\ell\geq 2$, then $s$ is of type D.
%\end{prop}

%\begin{thmp}

%\end{thmp}

\subsection{Initial considerations}
%It is easy to see that for $d_1=p-1$, we have $d_{p-1}=d_{p}=3$, i.e. $d\geq 2$:

%We have the following generalisation of the fact that $d\geq 2$.

As mentioned in the Introduction, $a\geq 2$ (recall \eqref{eqn:a}), and actually we can show rather more.
\begin{lemma}
	\label{le:SG}
	If $F$ is a radius one $3$-polytope, then
	\begin{equation}
	\label{eqn:d1}
	a\geq 2+\sum_{i\geq 3}(i-2)\cdot B_G(i),
	\end{equation}
	where
	\begin{equation*}
	B_G(i):=\#\{\text{blocks of } G\text{ bounded by an } i \text{-cycle}\}.
	\end{equation*}
\end{lemma}
\begin{proof}
We begin by showing that $a\geq 2$. As the graph $G$ is non-empty, we may take $e_1=uv$ to be any edge. The claim is, at least one vertex on each of the two $uv$-paths of the cycle $H$ has degree $0$ in $G$. To see this, fix any $w\neq u,v$ on one of the $uv$-paths ($w$ exists as $e_1=uv$ is an edge in $G$ so it cannot be an edge in $H$). Either $\deg_{G}(w)=0$, or there is an edge $e_2=wx\in E(G)$. Now $G$ is planar so that $e_1,e_2$ cannot cross. We then consider the $wx$-subpath of the initial $uv$-path, and reason as above to conclude that since $G$ is finite, we indeed have $a\geq 2$.
	
Now let $G$ contain exactly one block, bounded by an $i$-gon, say. This determines $i$ (internally disjoint) $u_1u_2$-, $u_2u_3$-, $\dots$, $u_{i}u_1$- paths in $H$, hence reasoning as above $a\geq i= 2+(i-2)$ and \eqref{eqn:d1} is proven in the case of one block.

Next, call $J$ the subgraph of $G$ induced by vertices lying on cyclic blocks of $G$. We claim that there exists at least one block $B$ of $J$ with vertices all lying on the same path $u_1,u_2,\dots,u_n$ of $H$, such that $u_2,\dots,u_{n-1}$ belong to no other cyclic block of $J$. We start by checking if a block $B_1$ satisfies this property. If not, we can find an edge $e=w_1w_m$ on the boundary of $B_1$, such that $w_1,w_2,\dots,w_m$ are consecutive on $H$, no vertices from $w_2,\dots,w_{m-1}$ belong to $B_1$, and moreover by planarity of $J$ there is at least one cyclic block $B_2$ of $J$ with vertices a subset of the $w_1,w_2,\dots,w_m$. We repeat the above procedure with $B_2$ in place of $B_1$, and by finiteness of $J$ we eventually find at least a cyclic block $B$ with vertices all lying on the same path $u_1,u_2,\dots,u_n$ of $H$, such that $u_2,\dots,u_{n-1}$ belong to no other cycle of $J$.

Let $i$ be the length of the cycle bounding $B$. We denote its vertices in order along $u_1,u_2,\dots,u_n$ by
\[u_{k_{1}},u_{k_{2}},\dots,u_{k_{i}}.\]
Then there is a degree $0$ vertex of $G$ along $u_1,u_2,\dots,u_n$ between $u_{k_{l}}$ and $u_{k_{l+1}}$ for each $1\leq l\leq i-1$. We finally remove $B$ from $J$ so that we can argue that any such block added to $J$ increases $a$ by at least $(i-1)-1=i-2$, and the proof of the present lemma is complete.
%Since we already knew that there was at least one degree $0$ vertex along $e_1$, the new $j$-cycle has introduced at least $j-2$ new ones, hence the claim \eqref{eqn:d1} is proven.
\end{proof}

\subsection{Conventions}
%All graphs that appear are simple.
%All graphs that appear contain no loops and multiple edges.
Everywhere we fix the notation $F=(V,E)$, $|V|=p$ and $|E|=q$ for a $3$-polytope of radius $1$, order $p$, size $q$, and degree sequence $s$ \eqref{eqn:s}. We write $F_1\simeq F_2$ when $F_1,F_2$ are isomorphic graphs. The notation $H,G,s'$ will always be as in Definition \ref{def:G}.
%, and $H\prec G$ when $H$ is (isomorphic to) a subgraph of $G$.
%A graph of order $k+1$ or more is $k$-connected if removing any set of $k-1$ or fewer vertices leaves the graph connected.
%Regions of a $2$-connected planar graph are cycles of length $i$ ($i$-gons) \cite[Proposition 4.26]{dieste}. For these graphs, the terms `region' and `face' are interchangeable. The $i$-gonal faces will be denoted by their sets of $i$ vertices. If $[a,b,c]$ is a triangle, we call \textit{splitting} the operation of adding a vertex $d$ and edges $da, db, dc$.
%If $\{a,b,c\}$ is a triangular region of a planar graph, we call \textit{splitting} the operation of adding a vertex $d$ and edges $da, db, dc$. 
% The abbreviation w.l.o.g. means `without loss of generality'.

We use $K_p$ for the complete graph on $p\geq 1$ vertices, $C_p$ with $p\geq 3$ for a cycle, and $\calS_n$ for the star graph on $n\geq 2$ edges. A \textit{caterpillar} is a tree where if we delete all degree one vertices, we are left with a central path $c_1,c_2,\dots,c_\ell$, $\ell\geq 1$. The caterpillar depends only on $\ell$ and on $x_j:=\deg(c_j)$, $j=1,\dots,\ell$, and we will denote it by $\calC(x_1,\dots,x_\ell)$. A special case is the star, $\calC(x)=\calS_{x}$.%Moreover, $K_{m_1,m_2,\dots,m_n}$ is the $n$-partite complete on parts of cardinality $m_1,m_2,\dots,m_n\geq 1$, and $S_p=K_{1,p-1}$ is the star graph on $p\geq 2$ vertices.

Notation such as $G-v$, $G+v$, $F-\{e_1,e_2\}$ means removing/adding vertices/edges/sets of vertices or edges from/to a graph.% A cyclic graph is a graph containing a cycle. A cyclic block is thus any block other than $K_2$. A block of a graph $G$ with only one vertex separating in $G$ is called an endblock.

\subsection{Acknowledgements}
%The author wishes to thank two anonymous referees for helpful corrections on a previous version.
%\\
The author was supported by Swiss National Science Foundation project 200021\_184927.
\\
The data of Tables \ref{tab:1}, \ref{tab:2} and \ref{tab:3} was obtained thanks to Scientific IT \& Application Support
(SCITAS) High Performance Computing (HPC) for the EPFL community.

\subsection{Data availability statement}
The code to produce the data of Tables \ref{tab:1}, \ref{tab:2}, and \ref{tab:3} is available on request.
%All data generated or analy\usepackage{subcaption}sed during this study are included in this article.

%\section{$a=2$ and $a=3$}
%\label{sec:23}
%\subsection{$a=2$}
\section{$a=2$: proof of Theorem \ref{prop:1}}
\label{sec:2}
In this section we assume that the number of degree three vertices in $F$ is $a=2$. All figures in this paper are sketches for the graph $F-v$. Its Hamiltonian cycle $H$ is the external cycle.
\begin{lemma}
	\label{le:forest}
The graph $G$ is a forest, and every connected component of $G$ is a caterpillar.
\end{lemma}
\begin{proof}
The first statement is trivial in light of Lemma \ref{le:SG}: the presence of a cycle would yield $a\geq 3$. For the second statement, again by contradiction, any non-caterpillar tree contains the subgraph $A$ depicted in Figure \ref{pic:001}. With the labeling as in Figure \ref{pic:001}, reasoning as in Lemma \ref{le:SG} there are degree $0$ vertices of $G$ along the internally disjoint $u_0u_1$- and $u_0u_3$-paths in $H$. Moreover, by planarity $u_4$ lies on one of the internally disjoint $a_1a_2$- and $a_2a_3$-paths, w.l.o.g. say the $a_2a_3$-path. Then there is a third degree $0$ vertex of $G$ along the $a_2a_4$-subpath.
\begin{figure}[h!]
	\centering
	\begin{subfigure}{.32\textwidth}
		\centering
		\includegraphics[width=3.5cm,clip=false]{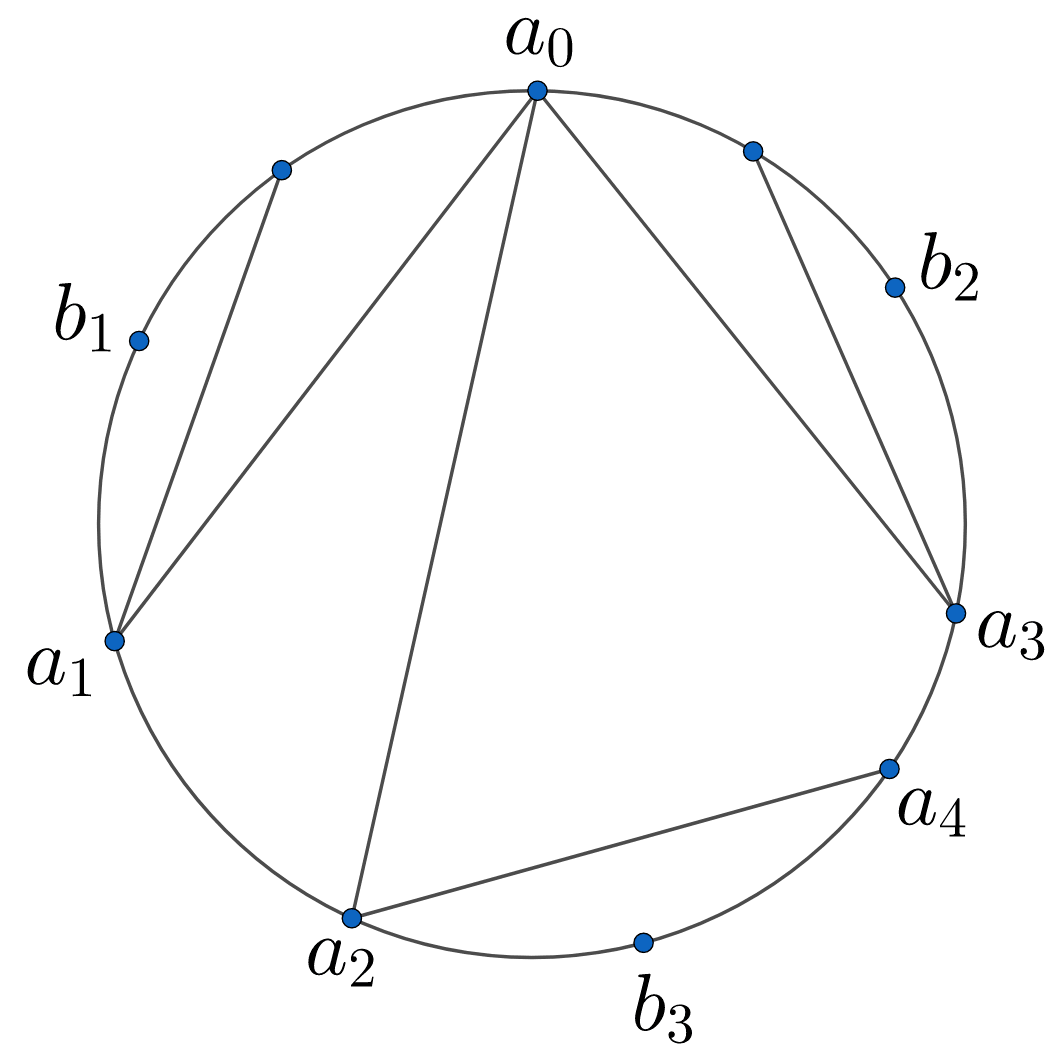}
		\caption{Union of Hamiltonian cycle $H$ and non-caterpillar $A$. There are at least three degree $0$ vertices in $G$, namely $b_1,b_2,b_3$.}
		\label{pic:001}
	\end{subfigure}
	\begin{subfigure}{.32\textwidth}
		\centering
		\includegraphics[width=2.5cm,clip=false]{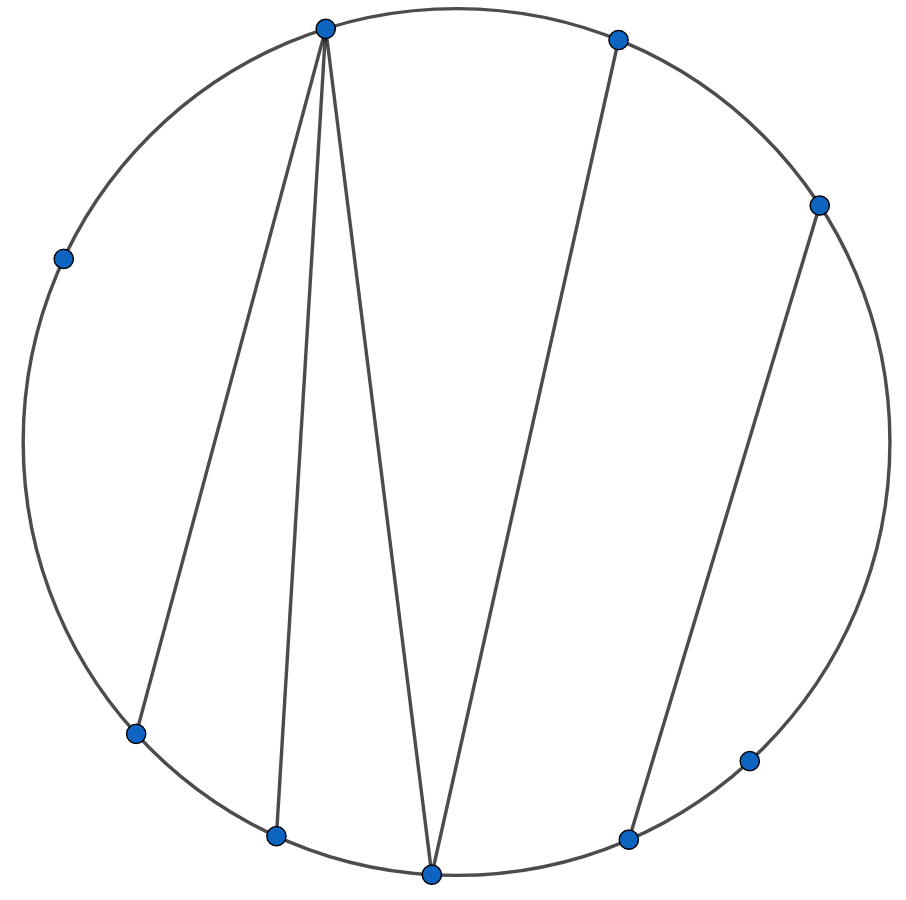}
		\vspace{0.5cm}\\
		\includegraphics[width=2.5cm,clip=false]{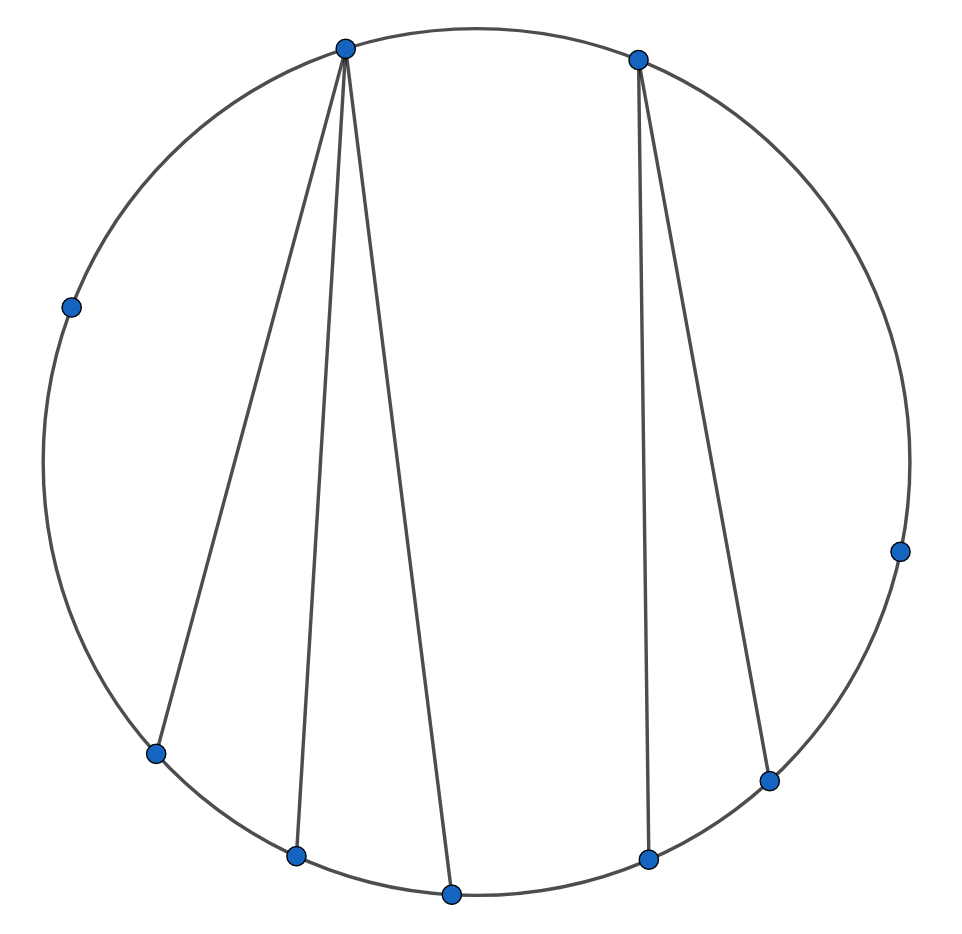}
		\caption{Two realisations of the sequence $3,2,1^5,0^2$.}
		\label{pic:002}
	\end{subfigure}
	\begin{subfigure}{.32\textwidth}
		\centering
		\includegraphics[width=3cm,clip=false]{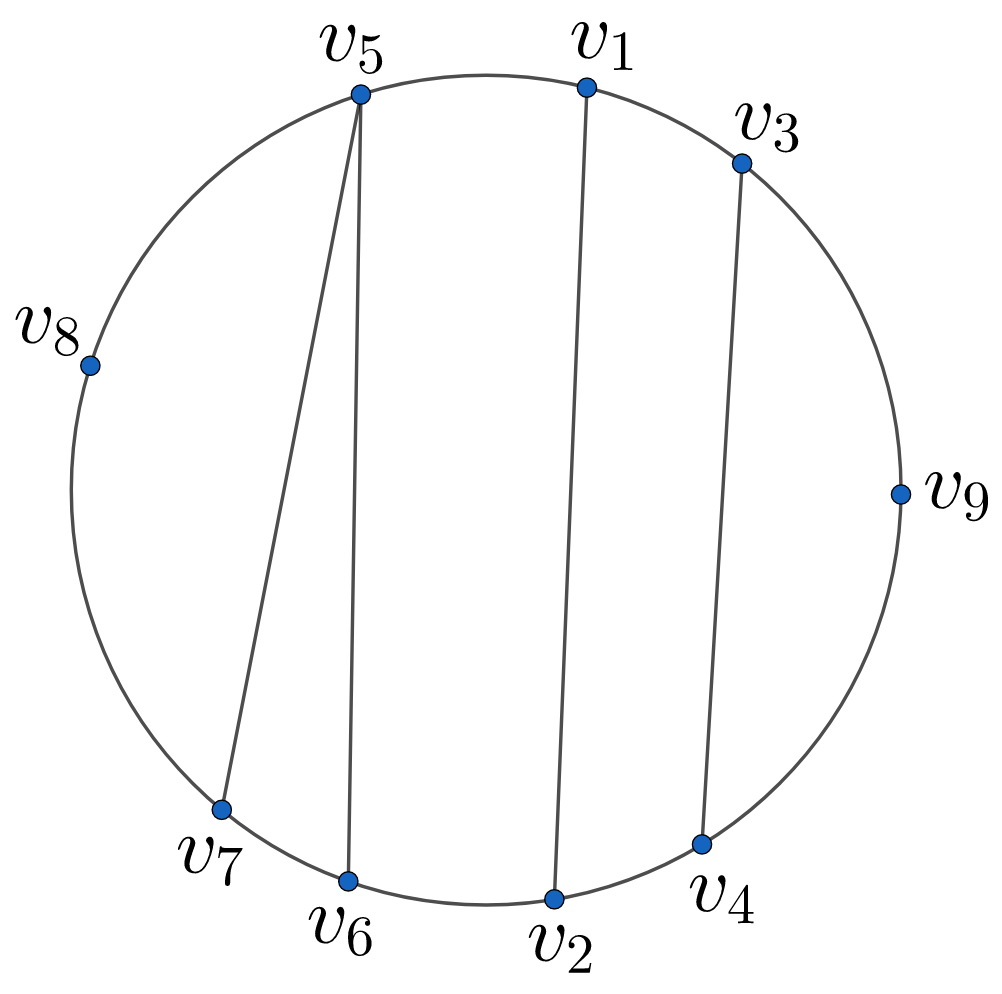}
		\vspace{0.25cm}\\
		\includegraphics[width=3cm,clip=false]{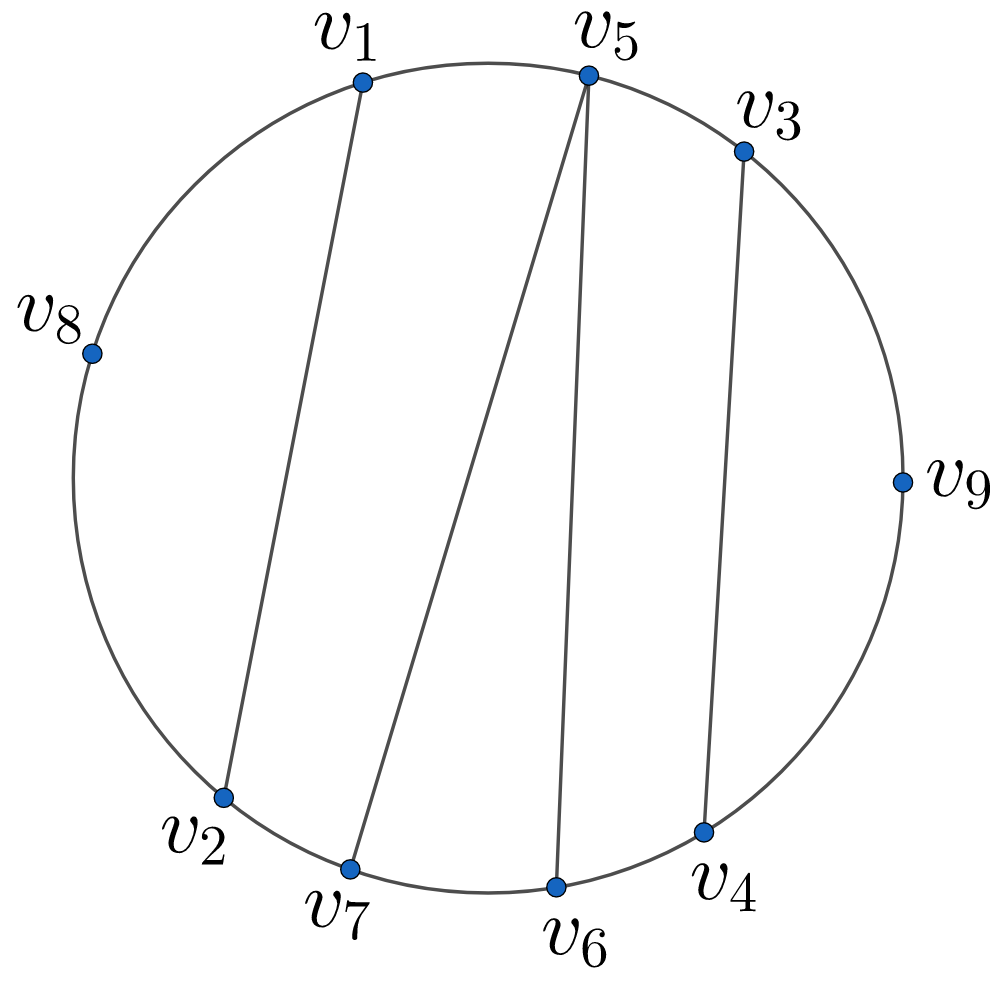}
		\caption{Two realisations of the sequence $2,1^6,0^2$. Here $p=10$ and $t=4$.}
		\label{pic:004}
	\end{subfigure}
	\caption{$a=2$.}
\end{figure}
%would increase $a$ strictly -- 
\end{proof}

We are reduced to characterising the sequences $s'$ with unique realisation as a disjoint union of caterpillars. Since $G$ is a forest, its sequence $s'$ has at least $2k$ $1$'s, where $k$ is the number of non-trivial components (these are non-trivial trees). Now if $s'$ is graphic, then it definitely has a realisation where $k-1$ of these components are just copies of $K_2$. Denote by $T$ the remaining component. For $s$ unigraphic, we certainly have the possibility that $G$ is just the union of $k\geq 1$ copies of $K_2$ and two isolated vertices, i.e. $s$ is of type A1, with $k=(p-1-2)/2$, i.e. $p\geq 5$ and odd.

If $k\geq 2$ but $T\not\simeq K_2$, then we claim that $T\simeq\calS_n$ is a star on $n\geq 2$ edges. To see this, suppose by contradiction that there are in $s'$ at least two values $x,y\geq 2$. Then we would have the two realisations $\calC(x,y)\cup K_2\cup G'$ and $\calS_x\cup \calS_y\cup G'$, where $G'$ is a forest on $k-2$ non-trivial trees, and $\calS$, $\calC$ denote stars and caterpillars -- refer to Figure \ref{pic:002}. Thus indeed $T\simeq\calS_n$, $n\geq 2$. Now we observe that in this case necessarily $k=2$: suppose not for a contradiction. We label 
\[v_1v_2, v_3v_4, \dots, v_{t-1}v_t\]
the copies of $K_2$, $v_{t+1}$ the centre of the star, $v_{t+2},\dots,v_{p-3}$ the remaining vertices of the star, and $v_{p-2},v_{p-1}$ the degree $0$ vertices, where $t$ is even and satisfies $4\leq t\leq p-6$. Then we would have two non-isomorphic realisations of $F$, one with the vertices
\[v_{p-2},v_{t+1},v_1,v_3,\dots,v_{t-1},v_{p-1},v_t,\dots,v_4,v_2,v_{t+2},\dots,v_{p-3}\]
in order around $H$, and another with the order
\[v_{p-2},v_1,v_{t+1},v_3,\dots,v_{t-1},v_{p-1},v_t,\dots,v_4,v_{t+2},\dots,v_{p-3},v_2\]
(Figure \ref{pic:004}). To summarise, $G$ is the disjoint union of two isolated vertices, a copy of $K_2$, and a star $\calS_n$, $n\geq 2$: this is possibility A2.

We are left with the case $k=1$, i.e. $G$ itself is a caterpillar together with two isolated vertices. Now any caterpillar $\calC(x_1,\dots,x_\ell)$ is determined by its length $\ell\geq 1$ and degrees $x_i\geq 2$, $1\leq i\leq\ell$, of vertices along its path. If $x_1\neq x_j$ for any $2\leq j\leq\ell$, we only need to exchange the order of the corresponding vertices along the path to obtain another realisation of $s$, \textit{except if} $\mathit{\ell=2}$. Thereby, either the $x_i$'s are all equal -- type A3, or $\ell=2$ -- type A4.

It is straightforward to check that, on the other hand, the sequences A1, A2, A3, and A4 are indeed unigraphic. This concludes the proof of Theorem \ref{prop:1}.
\begin{rem}
	\label{rem:cat}
	Let $Q_i$ be the set of degree one vertices adjacent to $u_i$ of degree $x_i$ in the caterpillar $\calC(x_1,\dots,x_\ell)$ of $G$. To ensure that $a=2$, the vertices of $\calC(x_1,\dots,x_\ell)$ are placed around $H$ in the order \[Q_1,x_1,Q_2,x_3,Q_4,x_5,\dots,Q_{\ell-2},x_{\ell-1},Q_{\ell},x_{\ell},Q_{\ell-1},x_{\ell-2},\dots,x_4,Q_3,x_2\]
	if $\ell$ is even, and
	\[Q_1,x_1,Q_2,x_3,Q_4,x_5,\dots,Q_{\ell-1},x_{\ell},Q_{\ell},x_{\ell-1},Q_{\ell-2},x_{\ell-3},\dots,x_4,Q_3,x_2\]
	if $\ell$ is odd.
\end{rem}

\section{$a=3$: Proof of Theorem \ref{prop:2}}
\label{sec:3}
In this section we assume that the number of degree three vertices in $F$ is $a=3$. Recall Definition \ref{def:G} for $H,G,s'$.

\begin{lemma}
	\label{le:d3}
	If $s$ is unigraphic, $a=3$, and $p\geq 7$, then $G$ has exactly one cycle, and this cycle is a triangle. Moreover, if we remove the three isolated vertices from $G$, the resulting graph is connected.
\end{lemma}
\begin{proof}
	%Thanks to Lemma \ref{le:SG}, we only need to show that $G$ is not a forest. Suppose by contradiction that it is. Every tree degree sequence admits a realisation as a caterpillar. Therefore, if $s$ is unigraphic, and $G$ a forest, then every tree of $G$ is a caterpillar. By the arguments of section \ref{sec:2}, the non-trivial connected components of $G$ are: either a certain number of copies of $K_2$; or one copy of $K_2$ and a star on at least three vertices; or one caterpillar. We now have one more isolated vertex than for the case $a=2$, and we inspect the above possibilities for $G$ (e.g. Figure \ref{pic:006}) to conclude that the only way for $s$ to be unigraphic when $a=3$ and $G$ is a forest, is the sequence $5,4,4,3,3,3$ -- Figure \ref{pic:007}. The first statement of this lemma has thus been proven.
	Thanks to Lemma \ref{le:SG}, we only need to show that $G$ is not a forest. Suppose by contradiction that it is. Every tree degree sequence admits a realisation as a caterpillar. Therefore, if $s$ is unigraphic, and $G$ a forest, then every tree of $G$ is a caterpillar. 
	%By the arguments of section \ref{sec:2}, the non-trivial connected components of $G$ are: either a certain number of copies of $K_2$; or one copy of $K_2$ and a star on at least three vertices; or one caterpillar. We now have one more isolated vertex than for the case $a=2$, and we inspect the above possibilities for $G$
	Call $b_1,b_2,b_3$ the isolated vertices of $G$. If $s$ is unigraphic, then there is a unique realisation of $G-b_3$. However, as soon as $p\geq 7$, there are in $G$ at least three non-isolated vertices $a_1,a_2,a_3$, say. We may re-insert $b_3$ on any of the three disjoint $a_1a_2$-, $a_2a_3$-, $a_3a_1$-paths in $H$. One of these three paths will contain $b_1$, and at least one of the others does not contain $b_2$. There is thus more than one choice to place $b_3$, no $s$ is not unigraphical (e.g. Figure \ref{pic:006}). If $p\leq 6$, we readily find the only way for $s$ to be unigraphic when $a=3$ and $G$ is a forest, i.e. the sequence $5,4^2,3^3$ -- Figure \ref{pic:007}. The first statement of this lemma has thus been proven.
	\begin{figure}[h!]
		\centering
		\begin{subfigure}{.35\textwidth}
			\centering
			\includegraphics[width=3.5cm,clip=false]{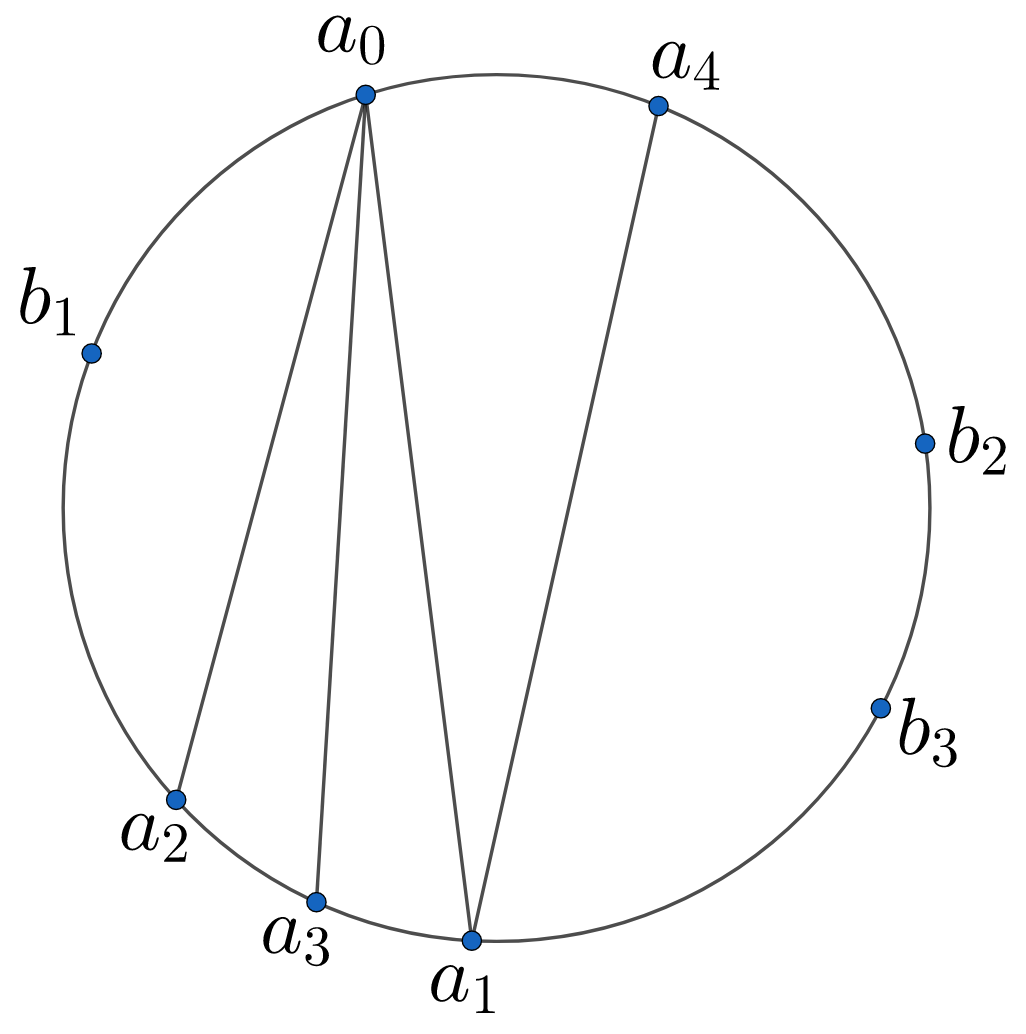}
			\caption{Although $7,6,5,4^3,3^2$ is unigraphic (type A4, $p=8$, $x=4$), the depicted $8,6,5,4^3,3^3$ is not. For instance, moving $b_3$ to the shortest $a_0a_4$-path yields a non-isomorphic realisation.}
			\label{pic:006}
		\end{subfigure}
		\hspace{1cm}
		\begin{subfigure}{.19\textwidth}
			\centering
			\includegraphics[width=2.75cm,clip=false]{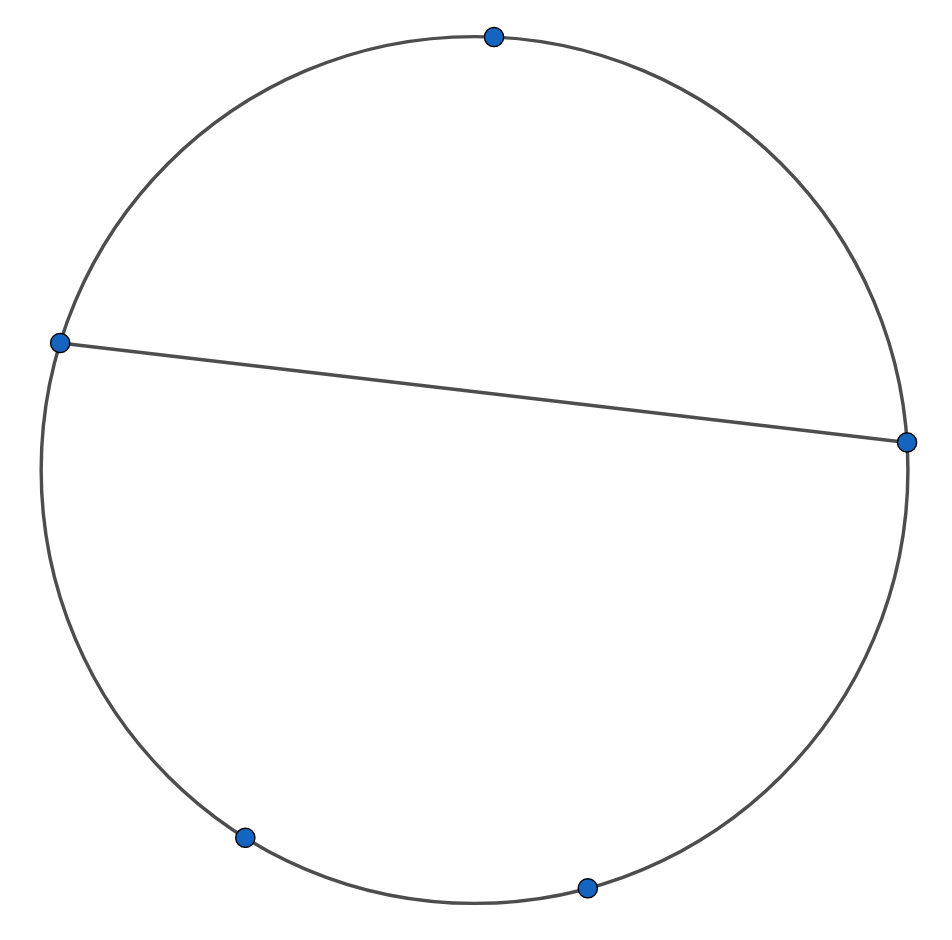}
			\caption{The unigraphic $5,4^2,3^3$.}
			\label{pic:007}
		\end{subfigure}
		\hspace{1cm}
		\begin{subfigure}{.29\textwidth}
			\centering
			\includegraphics[width=3.25cm,clip=false]{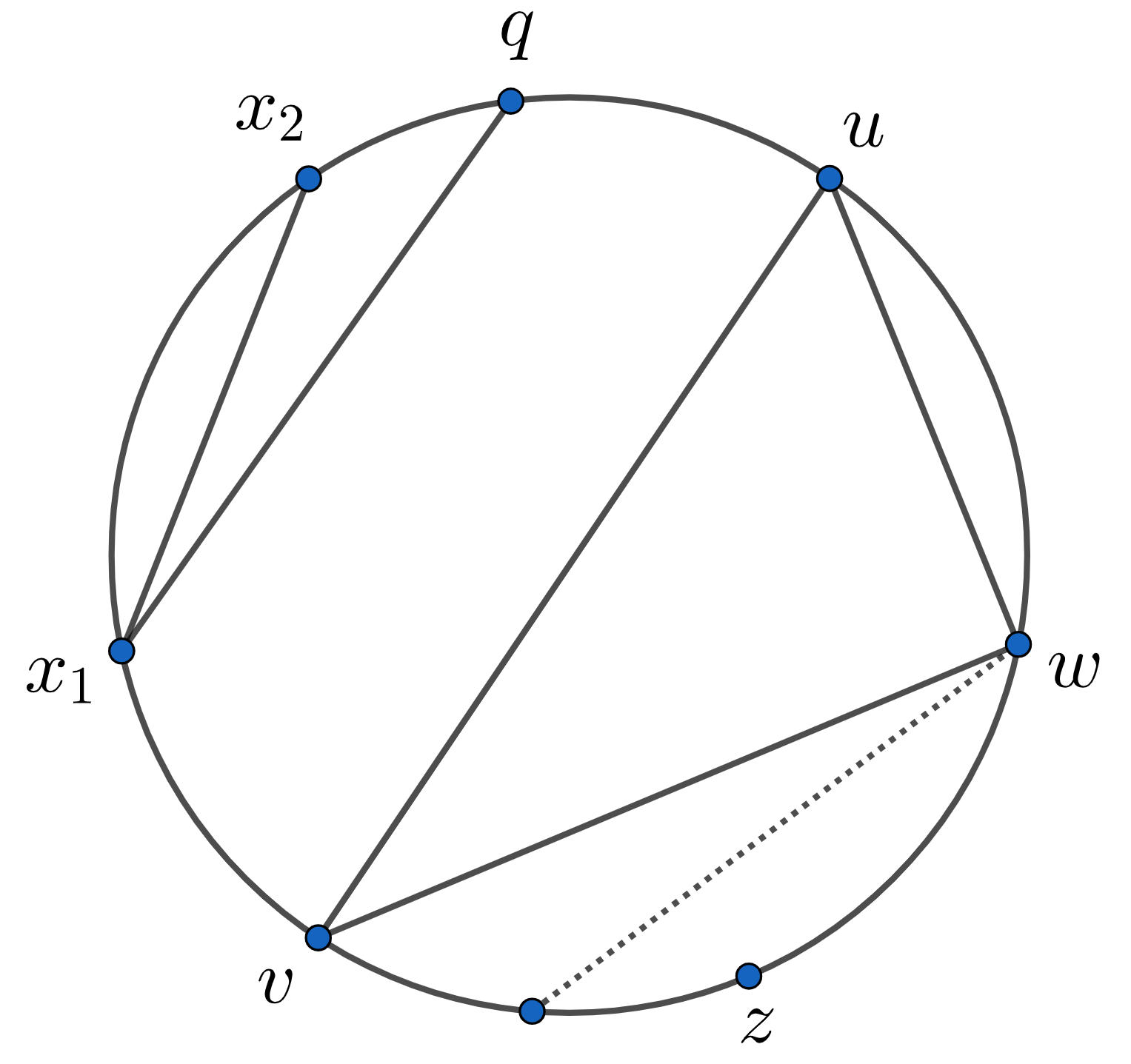}
			\includegraphics[width=3.25cm,clip=false]{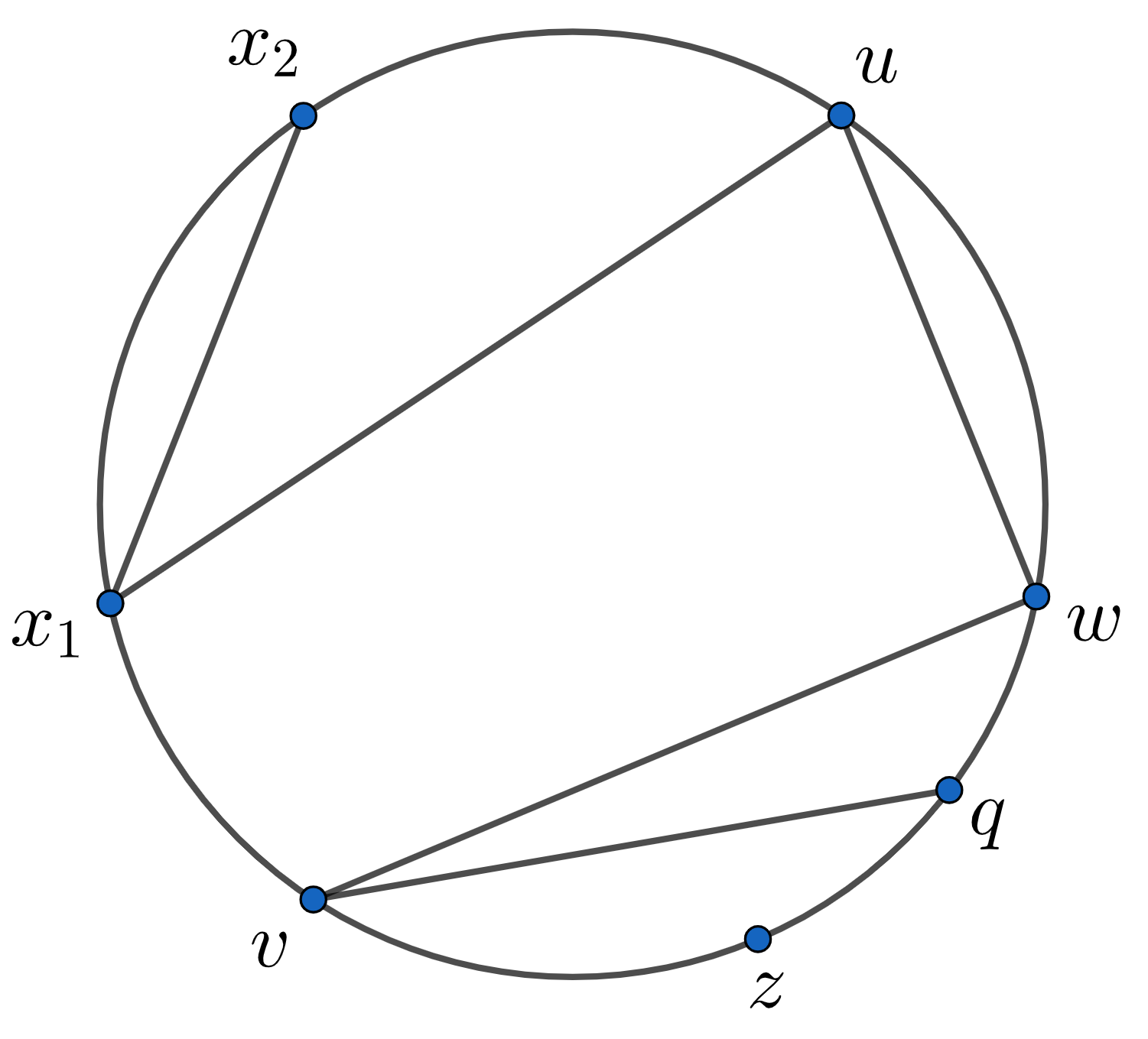}
			\caption{Proof of Lemma \ref{le:d3}, second statement.}
			\label{pic:008}
		\end{subfigure}
		\caption{$a=3$, Lemma \ref{le:d3}.}
	\end{figure}
	
	Turning to the second statement, assume by contradiction that $G$ has more than one non-trivial component. The other components apart from the one containing the triangle are trees. Now we use the same method as in the case A2 in section \ref{sec:2} (c.f. Figure \ref{pic:004}), to see that there is at most one non-trivial tree, and as remarked above, it is a caterpillar by unigraphicity. Denote this caterpillar as usual by $\calC(x_1,\dots,x_\ell)$, and $qx_1\in E$, $\deg_{G}(q)=1$. The vertices of the triangle are $u,v,w$. Let $z$ be an isolated vertex of $G$ on the $vw$-path of $H$ not containing $u$. On this path, by planarity there cannot exist at the same time a vertex $u_1$ between $w$ and $z$ adjacent to $v$ in $G$, and a vertex $u_2$ between $v$ and $z$ adjacent to $w$ in $G$. W.l.o.g, the latter does not happen. We then arrange the various vertices around $H$ so that $q,u,w,v,x_1$ are in this order (Figure \ref{pic:008}, top). We perform the transformation
	\[G-qx_1+x_1u-uv+vq,\]
	with the re-ordering $u,w,q,v,x_1$ around $H$ (Figure \ref{pic:008}, bottom). We have obtained two non-isomorphic realisations of $s$, contradiction.
	%Due to Remark \ref{rem:cat}, we may assume that $v$ is adjacent to $x_1$ in $H$, and $u$ to a vertex of degree one adjacent to $x_1$, say $q\in Q_1$ to use the notation of Remark \ref{rem:cat}. We then delete $q$ from $G$ and add $x_1u$ instead. Call $z_1$ the isolated vertex of $G$ on the $vw$-path of $H$ not containing $u$. On this path, by planarity there cannot exist at the same time a vertex $u_1$ between $w$ and $z_1$ adjacent to $v$ in $G$, and a vertex $u_2$ between $v$ and $z_1$ adjacent to $w$ in $G$. W.l.o.g, the former does not happen. We then insert $q$ back in $G$ to be the vertex closest to $w$ on the aforementioned $vw$-path of $H$, we add the edge $vq$ to $G$, and delete $uv$ (e.g. Figure \ref{pic:008}). 
	%There is a realisation of $s'$ where this tree is $K_2$. That being the case, we now delete the $K_2$ and one edge $xy$ of the triangle. We then insert two bridges from $x,y$ to two new vertices (e.g. Figure \ref{pic:008}).
	
	%and $G$ a forest, then $a=2$. It is now easy to see that, due to Lemma \ref{le:SG}, if $a=3$ and $p\geq 7$, then $G'$ has exactly one cycle, and this cycle is a triangle. By inspection, one finds the unigraphic $5,4,4,3,3,3$ and $6,5,5,5,3,3,3$ for $p\leq 7$.
\end{proof}

By inspection, we find for $p=7$ the only solution $6,5^3,3^3$ (where $G$ is simply a triangle together with the three degree $0$ vertices). For the rest of this section assume that $p\geq 8$ and define the non-empty graph
\[G':=G-\{\text{edges of the triangle}\}.\]
By Lemma \ref{le:d3}, the non-trivial components of $G'$ are three trees, and more precisely they are caterpillars (cf. Lemma \ref{le:forest}). Denote by $u,v,w$ the vertices of $G'$ that belong to the triangle in $G$.

\begin{itemize}
	\item
	Let us first consider the scenario when only one of the three mentioned caterpillars is non-trivial, and call it $\calC(x_1,\dots,x_\ell)$. It is connected to vertex $u$ of the triangle, say, and $x_i=\deg(c_i)\geq 2$.
	
	\begin{itemize}
		\item
		Let $\ell=1$, i.e. the only non-trivial caterpillar is a star. If $u$ is at the centre of this star, we get more than one choice for where the other vertices of the star are located, the $uv$-path in $H$ not containing $w$, or the $uw$-path not containing $v$ -- Figure \ref{pic:009}. Only when the star is $K_2$ there is just one choice up to isomorphism, and this is type C1 for $p=8$ -- Figure \ref{pic:011}. If instead $u$ is a degree one vertex of the star, and the star is $\calS_2$, we have type C1 for $p=9$ -- Figure \ref{pic:012}. The centre $c$ cannot have degree $3$ or more: otherwise, we could remove all degree $1$ vertices adjacent to $c$, except $u$ and one more, and insert the same number of degree $1$ vertices adjacent to $v$ -- Figure \ref{pic:013}. This would yield another realisation of $s$.
		
		\begin{figure}[h!]
			\centering
			\begin{subfigure}{.24\textwidth}
				\centering
				\includegraphics[width=3cm,clip=false]{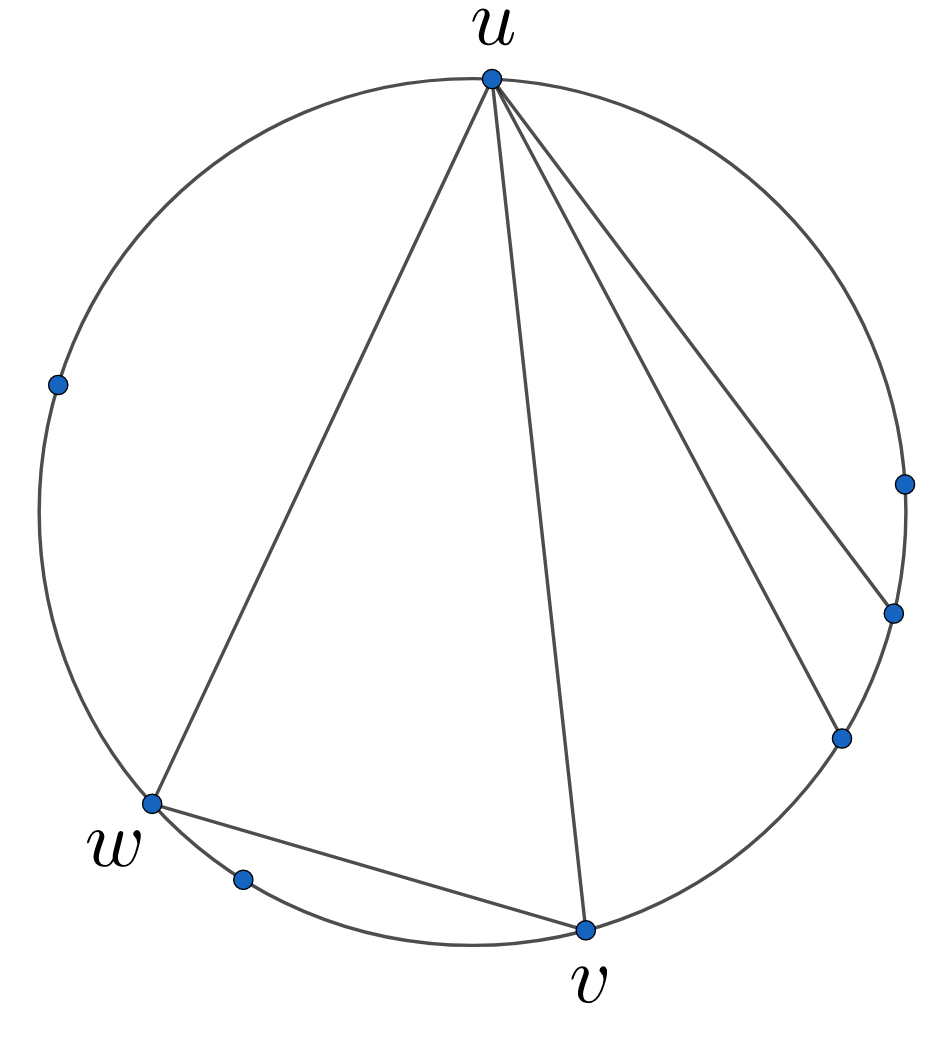}
				\includegraphics[width=3cm,clip=false]{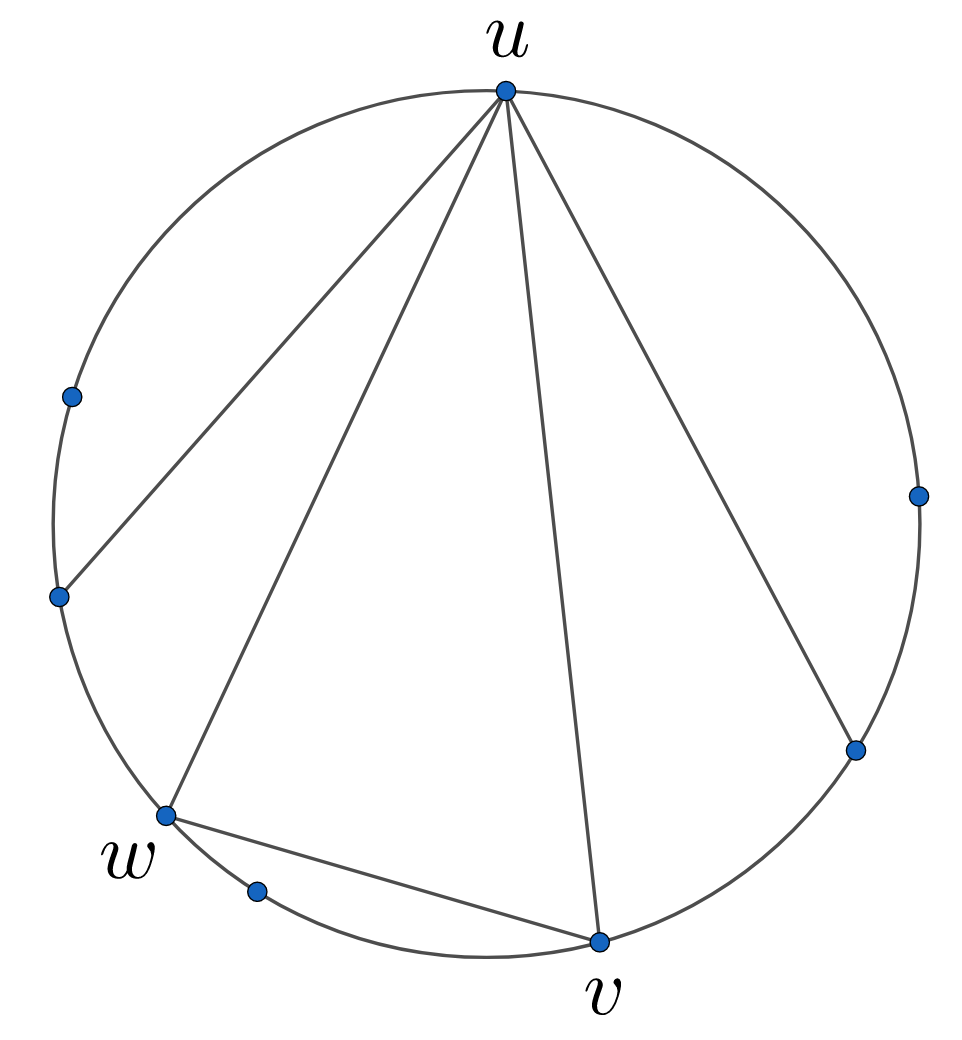}
				\caption{The vertex $u$ is at the centre of the star $\calS_2$.}
				%\caption{The non-unigraphic $8,7,5^2,4^2,3^3$.}
				\label{pic:009}
			\end{subfigure}
			\hspace{0.25cm}
			\begin{subfigure}{.20\textwidth}
				\centering
				\includegraphics[width=2.5cm,clip=false]{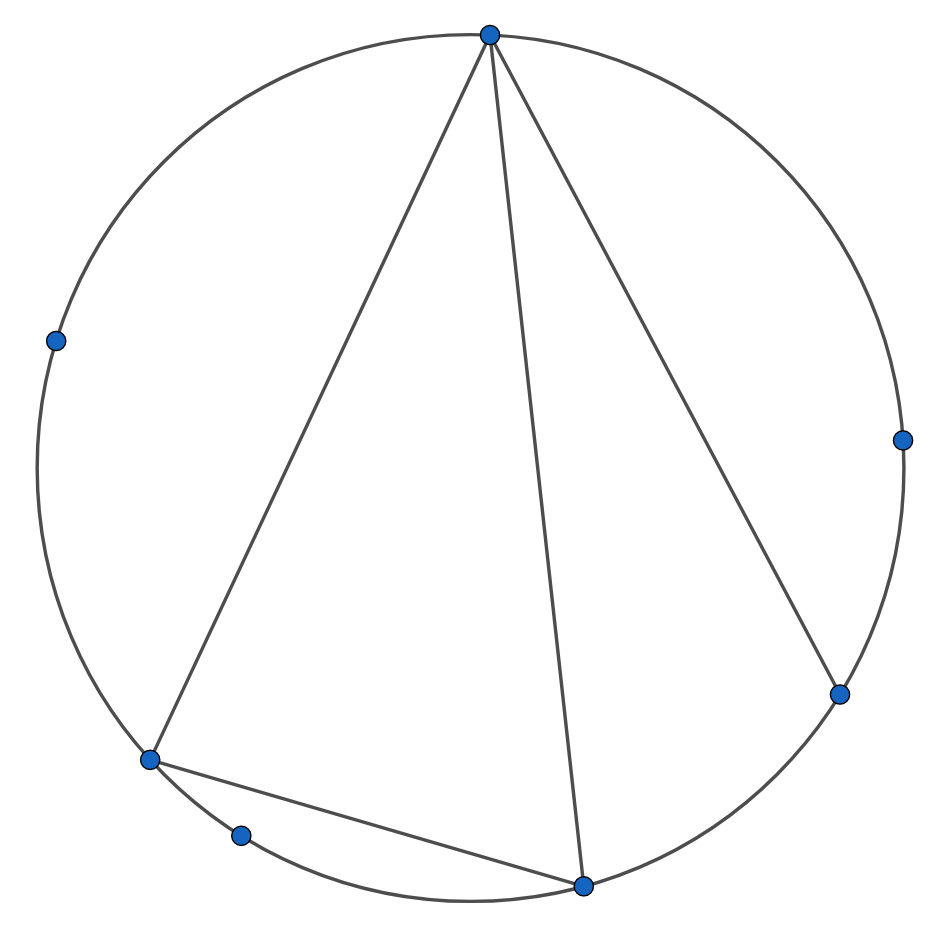}
				\caption{The unigraphic $7,6,5^2,4,3^3$ (type C1, $p=8$).}
				\label{pic:011}
			\end{subfigure}
			\hspace{0.25cm}
			\begin{subfigure}{.20\textwidth}
				\centering
				\includegraphics[width=2.5cm,clip=false]{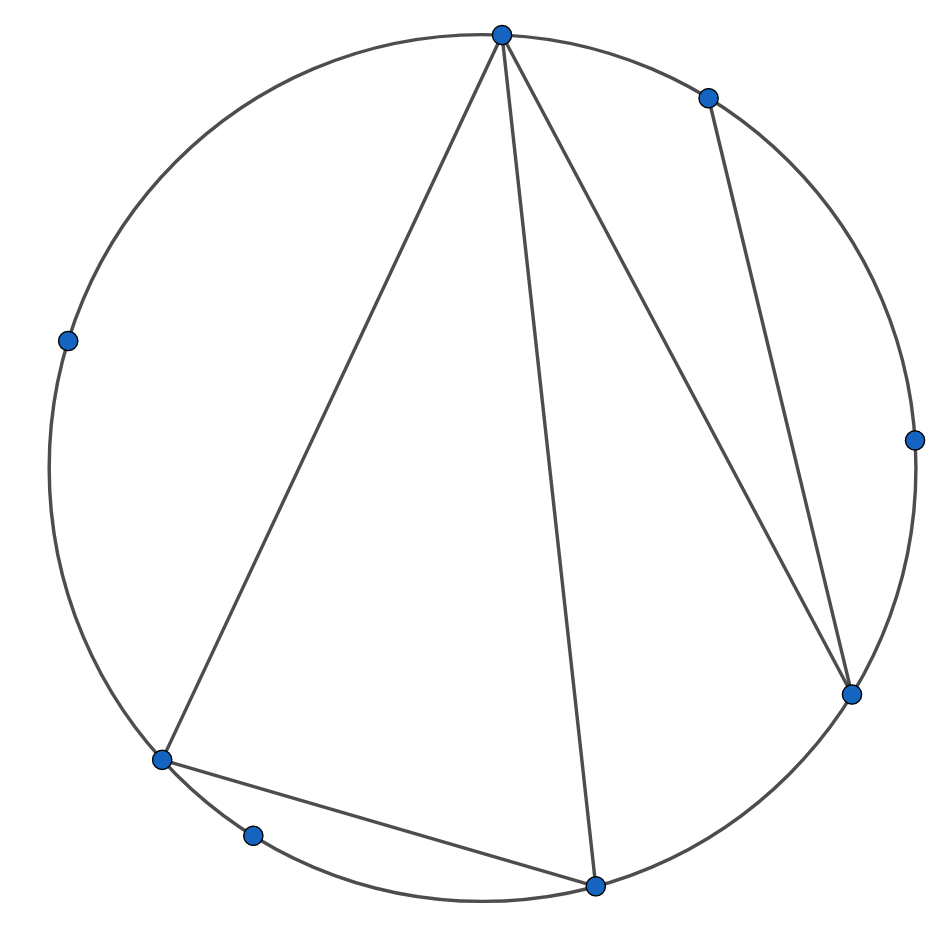}
				\caption{The unigraphic $8,6,5^3,4,3^3$ (type C1, $p=9$).}
				\label{pic:012}
			\end{subfigure}
			\hspace{0.25cm}
			\begin{subfigure}{.24\textwidth}
				\centering
				\includegraphics[width=3cm,clip=false]{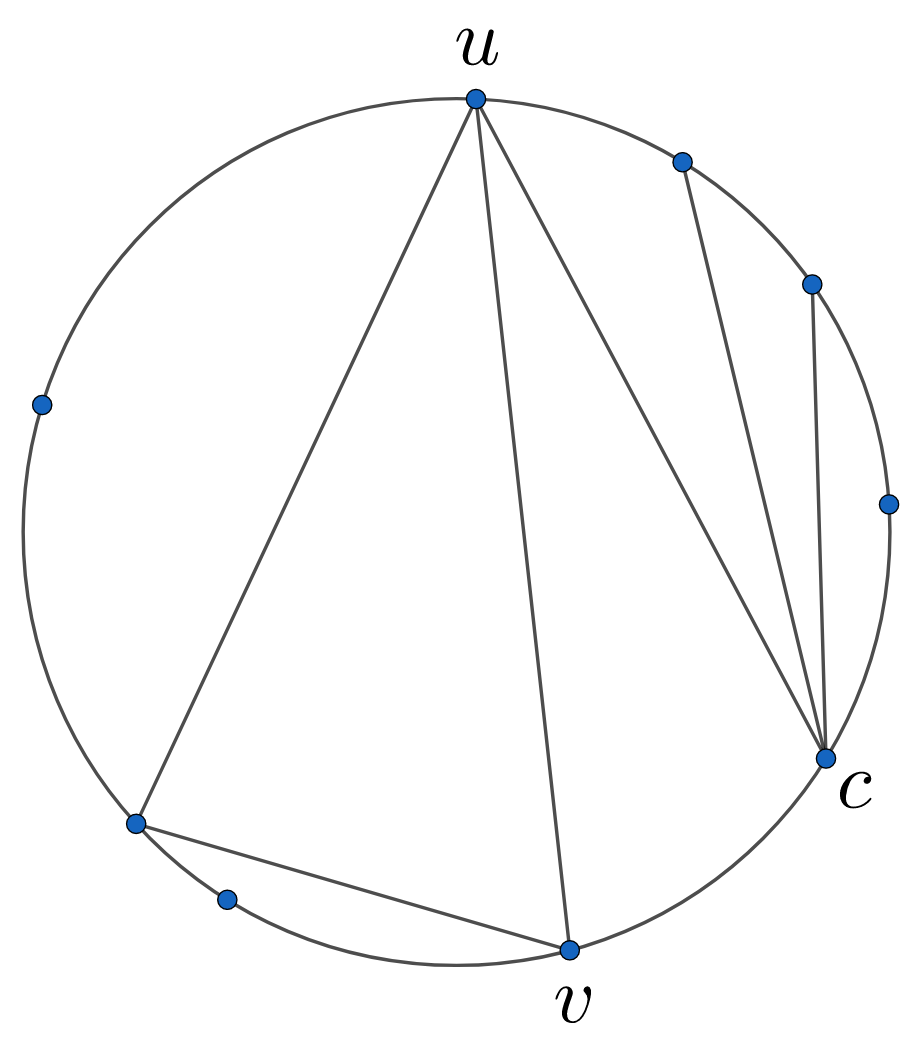}
				\includegraphics[width=3cm,clip=false]{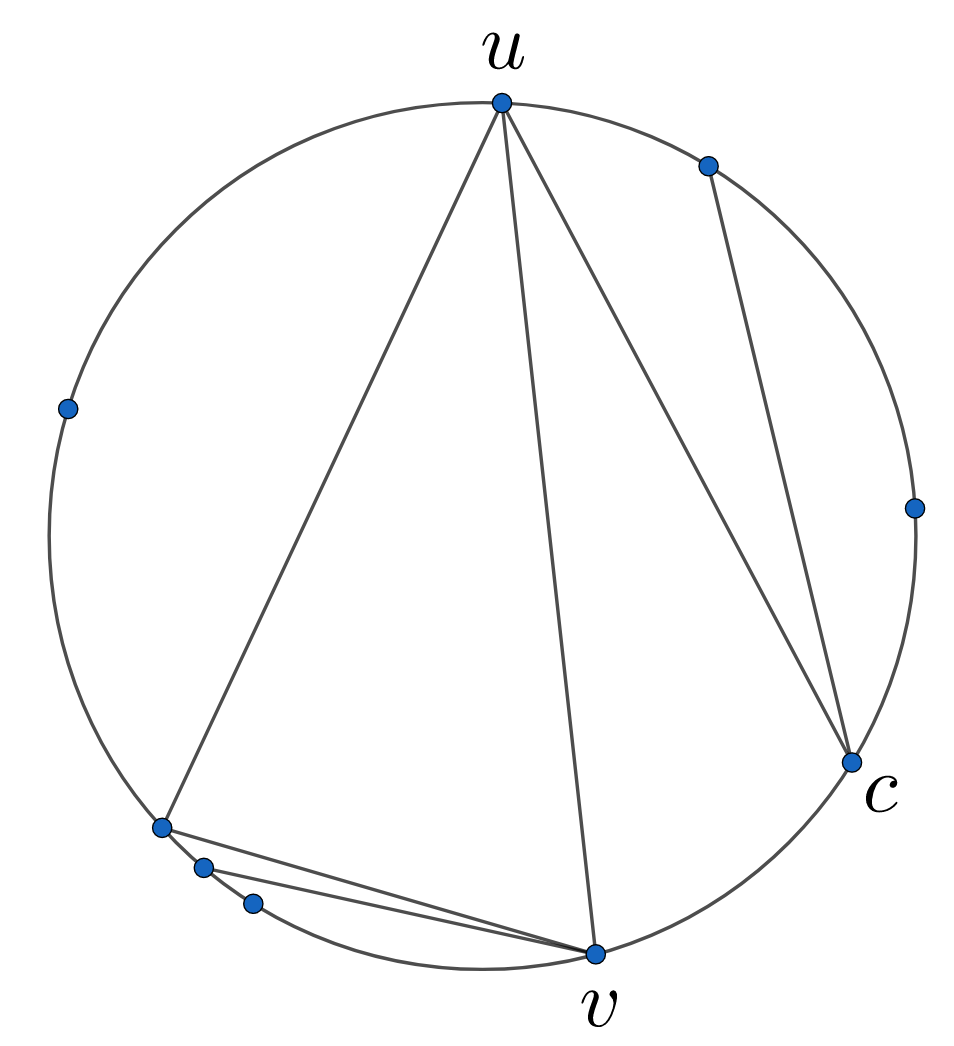}
				\caption{The non-unigraphic $9,6^2,5^2,4^2,3^3$.}
				\label{pic:013}
			\end{subfigure}
			\caption{$a=3$, case of one non-trivial caterpillar with $\ell=1$.}
		\end{figure}
		
		\item
		Now assume that $\ell\geq 2$. We argue that $u$ cannot be one of the vertices $c_i$ along the central path of the caterpillar $\calC(x_1,\dots,x_\ell)$: for $\ell\geq 3$, there would be a different realisation of $s$ as we simply move $u$ along this path -- Figure \ref{pic:015}. For $\ell=2$, we use the trick in Figure \ref{pic:017} to rule out $u=c_i$. Thereby, $u$ is a degree one vertex on $\calC(x_1,\dots,x_\ell)$.
		\\
		If in the caterpillar $\deg(c_i)=x_i\geq 3$ for some $1\leq i\leq\ell$, then similarly to Figure \ref{pic:013} we could remove all degree $1$ vertices adjacent to $c_i$ and insert the same number of degree $1$ vertices adjacent to $v$ to obtain another realisation of $s$. It follows that $x_i=2$ for all $1\leq i\leq\ell$, i.e. $\calC(x_1,\dots,x_\ell)=\calC(2,\dots,2)$, and as argued above, $u$ is at either end of this path. We get type C1 for $p\geq 10$. %Otherwise, $u$ coincides with one of the $x_i$, $1\leq i\leq\ell$. Here we get more than one realisation for $s$ -- Figure \ref{pic:015}, unless $\ell=2$. Then let $\ell=2$, w.l.o.g. $u=c_1$, and still w.l.o.g. $c_2$ is on the $uv$-path in $H$ not containing $w$. Then the degree $1$ vertex $a$ adjacent to $u$ could be either along the $uw$-path in $H$ not containing $v$, or along the $c_2v$-path not containing $w$ -- Figure \ref{pic:017}. These two choices mean that $s$ is not unigraphic.
		%Again by unigraphicity, $\deg_{G'}(u)=1$ -- see Figure. Moreover, if $\deg(c_i)=x_i\geq 3$ for some $1\leq i\leq\ell$, there would be another realisation of $G'$ where $\deg_{G'}(v)=x_i$ and $\deg_{G'}(c_i)=2$ -- see Figure. Therefore, $\calC(x_1,\dots,x_\ell)=\calC(2,\dots,2)$ is just a path -- type B1.% Furthermore, the vertices of the triangle are final vertices in the caterpillars
		%Now by unigraphicity, if one of the three mentioned caterpillars is the trivial graph, then exactly two of them are -- Figure. 
		\begin{figure}[h!]
			\centering
			\begin{subfigure}{.48\textwidth}
				\centering
				\includegraphics[width=2.75cm,clip=false]{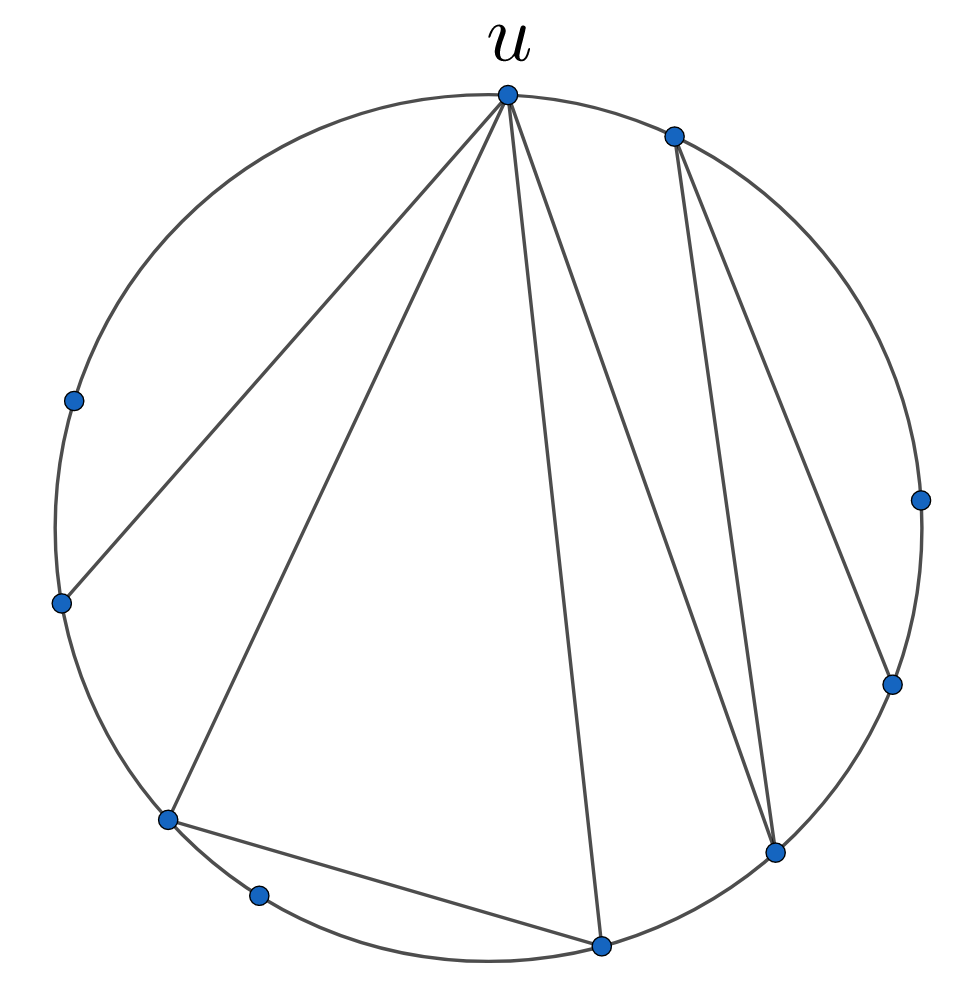}
				\hspace{0.15cm}
				\includegraphics[width=2.75cm,clip=false]{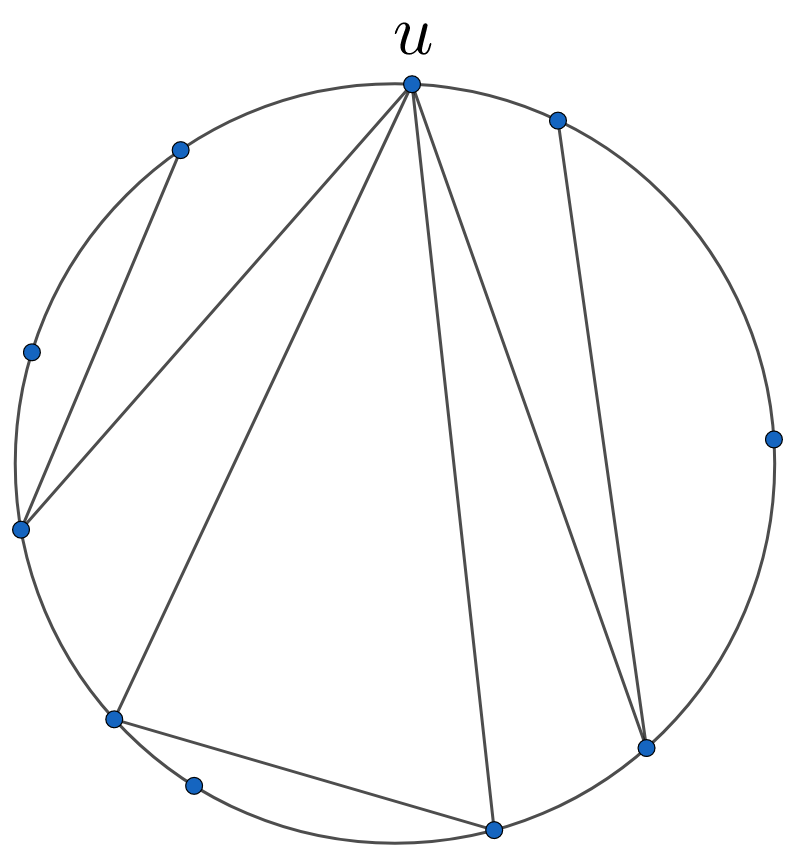}
				\caption{The case $u=c_i$, $\ell=3$: moving $u$ along the path of the caterpillar produces non-isomorphic graphs.}
				%\caption{The non-unigraphic $8,7,5^2,4^2,3^3$.}
				\label{pic:015}
			\end{subfigure}
			\hspace{0.25cm}
			\begin{subfigure}{.48\textwidth}
				\centering
				\includegraphics[width=2.75cm,clip=false]{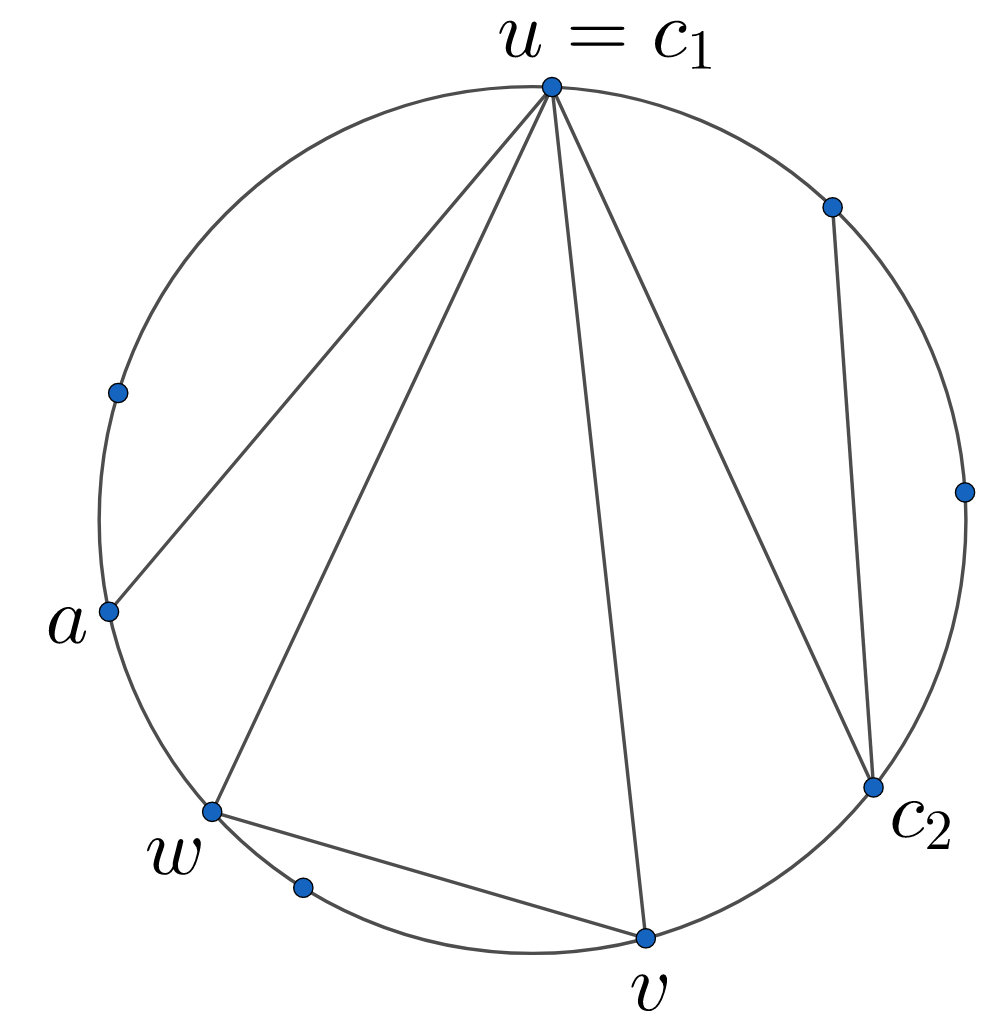}
				\hspace{0.1cm}
				\includegraphics[width=2.75cm,clip=false]{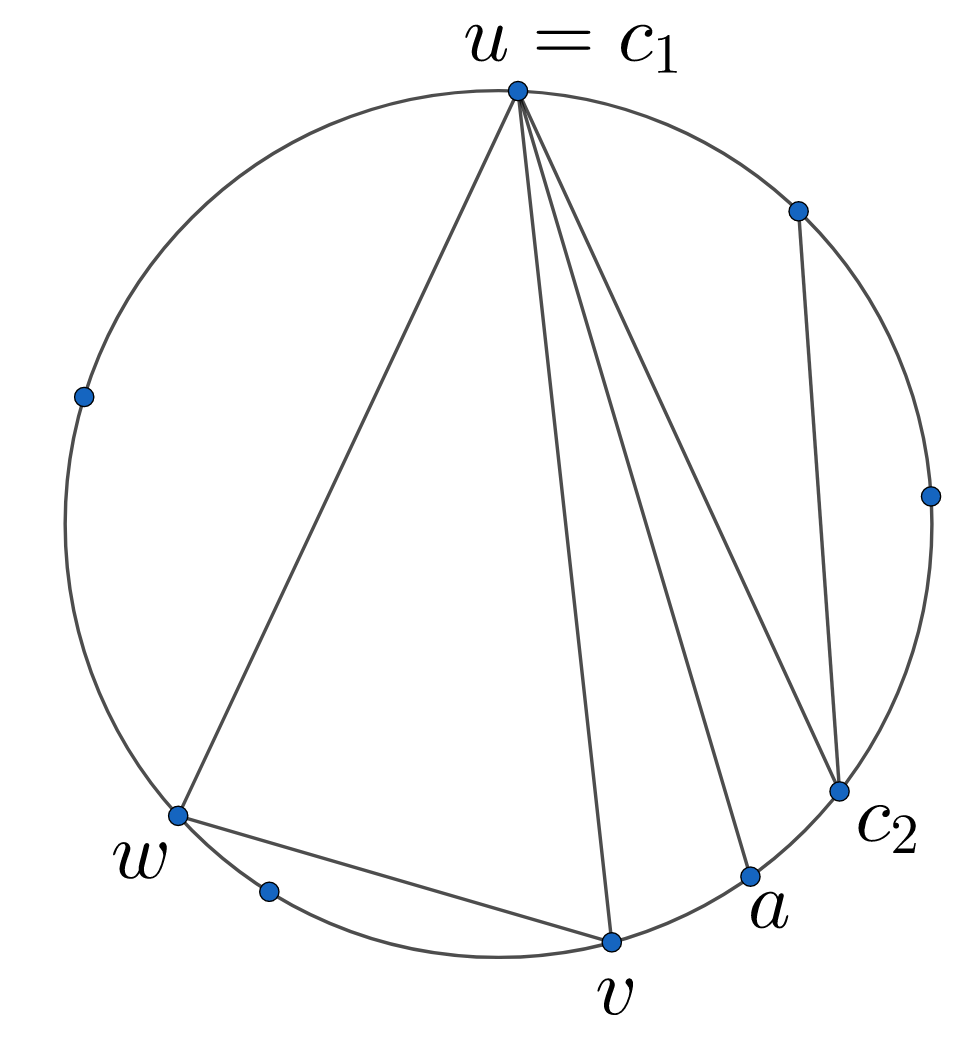}
				\caption{The case $u=c_i$, $\ell=2$: there are two choices for the vertex $a$ that give rise to non-isomorphic realisation of $s$.}
				\label{pic:017}
			\end{subfigure}
			\caption{$a=3$, one non-trivial caterpillar with $\ell\geq 2$.}
		\end{figure}
	\end{itemize}
	\item
	For the case of two non-trivial caterpillars, the above arguments imply that each is either a star centred at $u$ (resp. $v$) or a path with endpoint $u$ (resp. $v$). Similarly to the situation in Figure \ref{pic:009}, there is more than one choice for the subpath of $H$ where the other vertices of the caterpillars could lie.
	
	\item
	We next turn to the case of three non-trivial caterpillars, of lengths $\ell_1,\ell_2,\ell_3$, connected to vertices $u,v,w$ respectively. Similarly to the arguments of section \ref{sec:2}, we have $\ell_1+\ell_2+\ell_3\leq 4$ so that w.l.o.g. we can take $\ell_1\leq 2$ and $\ell_2=\ell_3=1$: otherwise, for $\ell_1+\ell_2+\ell_3=\ell\geq 5$ there would be the two possibilities $\ell_1=\ell-2$, $\ell_2=\ell_3=1$ or $\ell_1=\ell-3\geq\ell_2=2$, $\ell_3=1$, obtained by redistributing the vertices among the central paths of the three caterpillars. This leaves only two possibilities.
	
	\begin{itemize}
		\item
		First, if $\ell_1=\ell_2=\ell_3=1$, this means there are three stars, and by the above ideas these are centred at $u,v,w$. With no loss of generality, assume that
		\[\deg_{G}(u)\geq\deg_{G}(v)\geq\deg_{G}(w).\]
		For each degree $1$ vertex on the star centred at $u$, we have to choose if it lies on the $uv$-path in $H$ not containing $w$, or along the $uw$-path not containing $v$. However, they are all on the same subpath, otherwise $a\geq 4$ -- Figure \ref{pic:019}. Note that this situation arises because all three caterpillars are non-trivial. Further, if they are all on the former path, then all degree $1$ vertices on the star centred at $v$ lie on the $vw$-path not containing $u$, and all degree $1$ vertices on the star centred at $w$ lie on the $wu$-path not containing $v$ (Figure \ref{pic:020}, left). If we make the other choice to start with, we get another realisation of $s$ (Figure \ref{pic:020}, right), \textit{unless at least two of $\mathit{\deg_{G}(u),\deg_{G}(v),\deg_{G}(w)}$ are equal}. We have obtained a graphic sequence of type B1.
		
		\begin{figure}[h!]
			\centering
			\begin{subfigure}{.32\textwidth}
				\centering
				\includegraphics[width=3cm,clip=false]{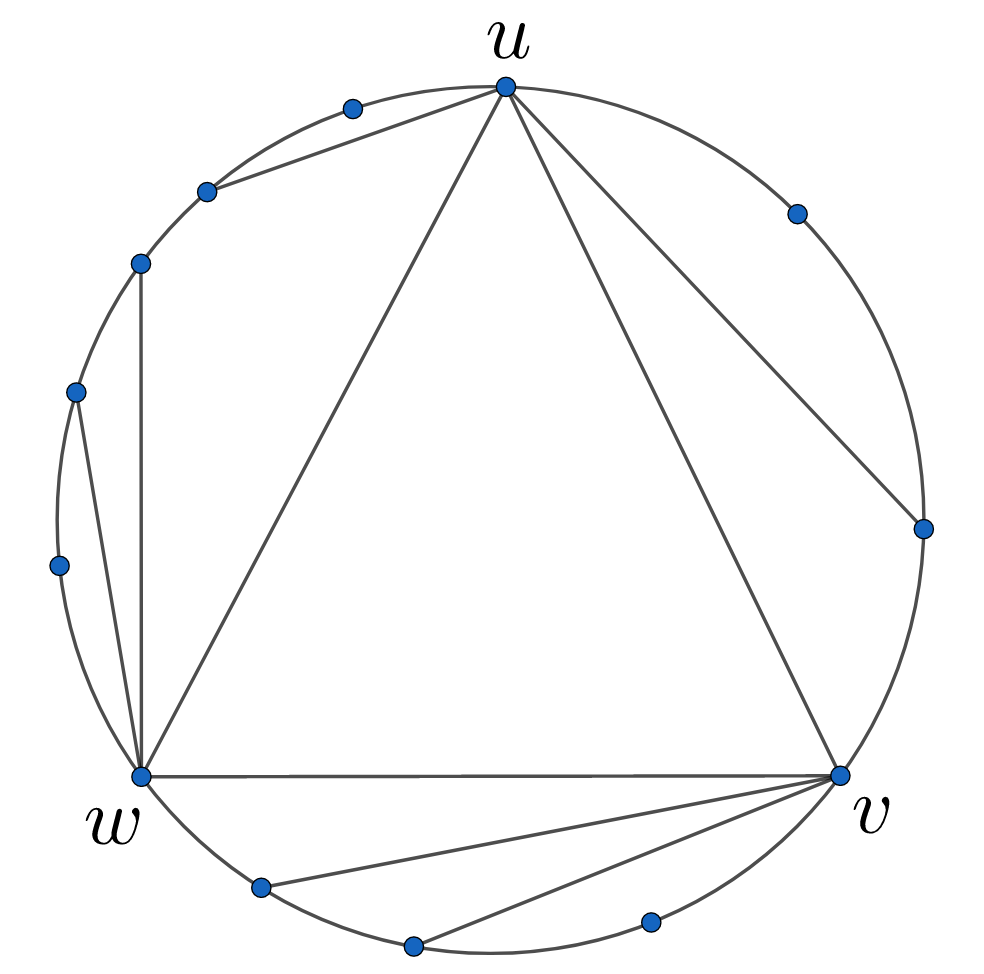}
				\caption{The degree $1$ vertices adjacent to $u$ are not all on the same $uv$- or $uw$-subpath. Thus $a\geq 4$.}
				%\caption{The non-unigraphic $8,7,5^2,4^2,3^3$.}
				\label{pic:019}
			\end{subfigure}
			\hspace{0.25cm}
			\begin{subfigure}{.64\textwidth}
				\centering
				\includegraphics[width=3cm,clip=false]{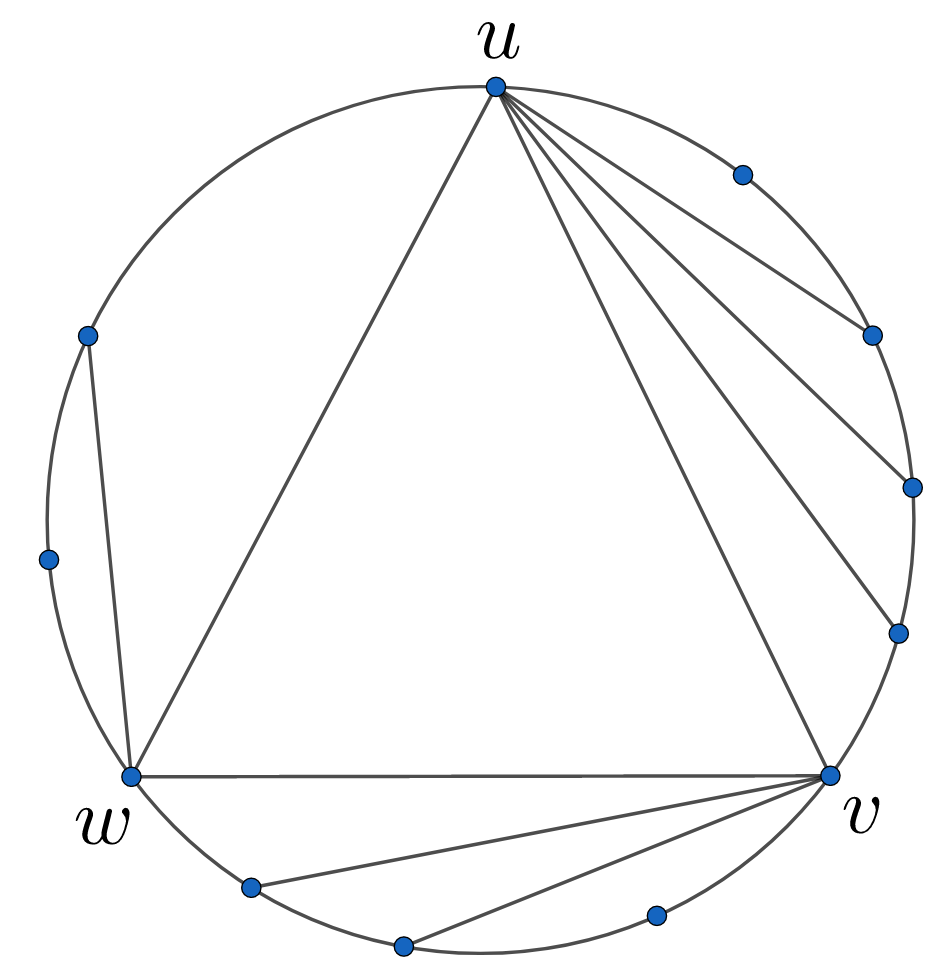}
				\hspace{0.4cm}
				\includegraphics[width=3cm,clip=false]{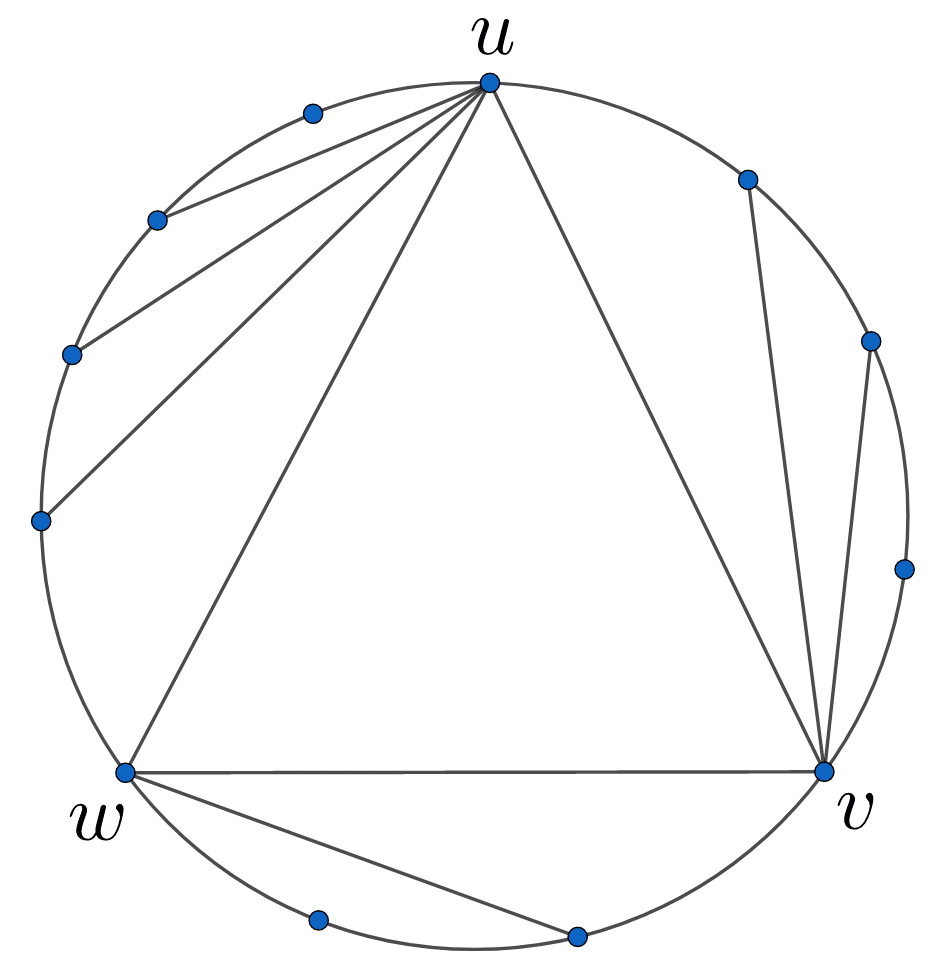}
				\caption{Here $\deg(u)>\deg(v)>\deg(w)$, thus $s$ is not unigraphic.}
				\label{pic:020}
			\end{subfigure}
			\caption{$a=3$, three non-trivial caterpillars, $\ell_1=\ell_2=\ell_3=1$.}
		\end{figure}
		
		\item
		Second and last, $\ell_1=2$, $\ell_2=\ell_3=1$, and by the above arguments $u=c_1$ along the length $2$ caterpillar. Then there are exactly four vertices in $G$ of degree $\geq 2$, namely $u,c_2,v,w$. We use unigraphicity to deduce that these four degrees must all be equal, yielding type D1 -- Figure \ref{pic:022}.
		
		\begin{figure}[h!]
			\centering
			\begin{subfigure}{.30\textwidth}
				\centering
				\includegraphics[width=3cm,clip=false]{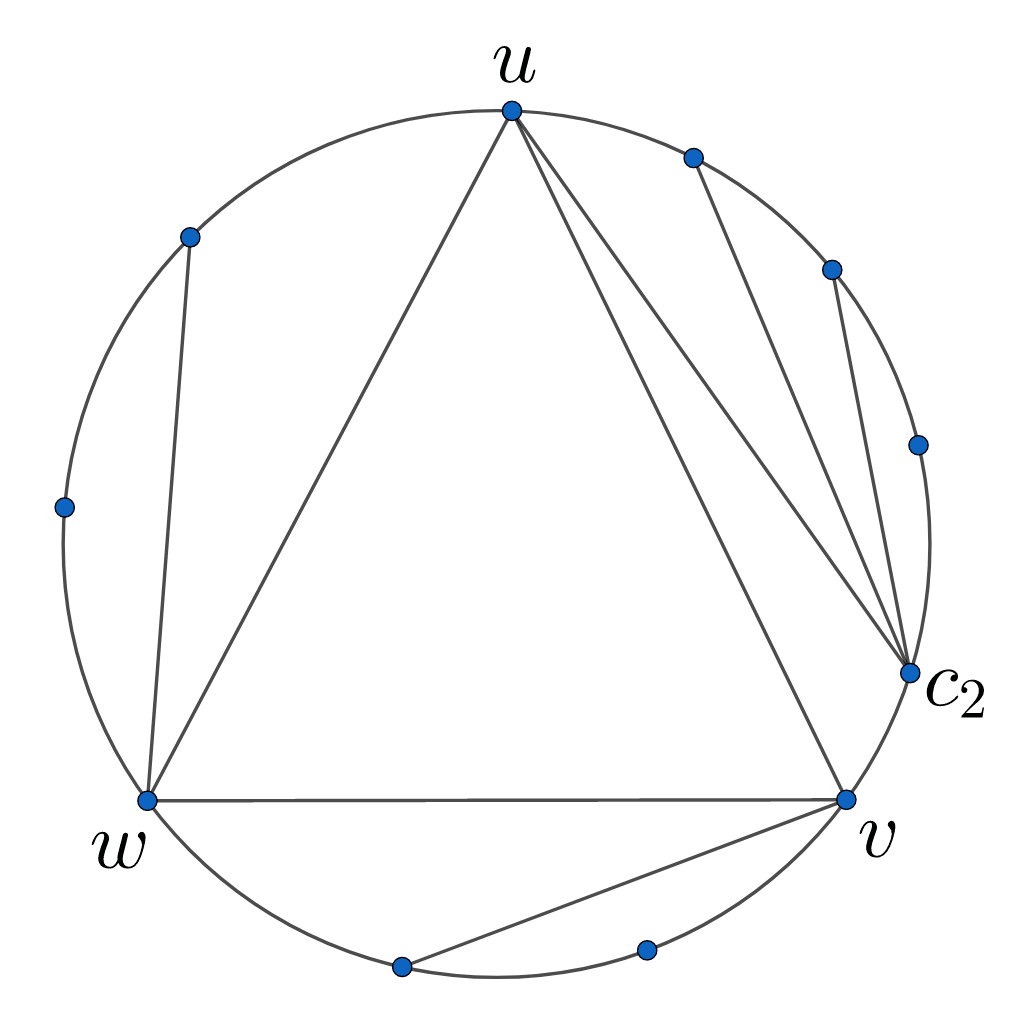}
				\caption{Here $\deg(u)=\deg(c_2)=\deg(v)=\deg(w)=3$, yielding type D1 for $p=12$.}
				\label{pic:022}
			\end{subfigure}
			\hspace{0.5cm}
			\begin{subfigure}{.62\textwidth}
				\centering
				\includegraphics[width=3cm,clip=false]{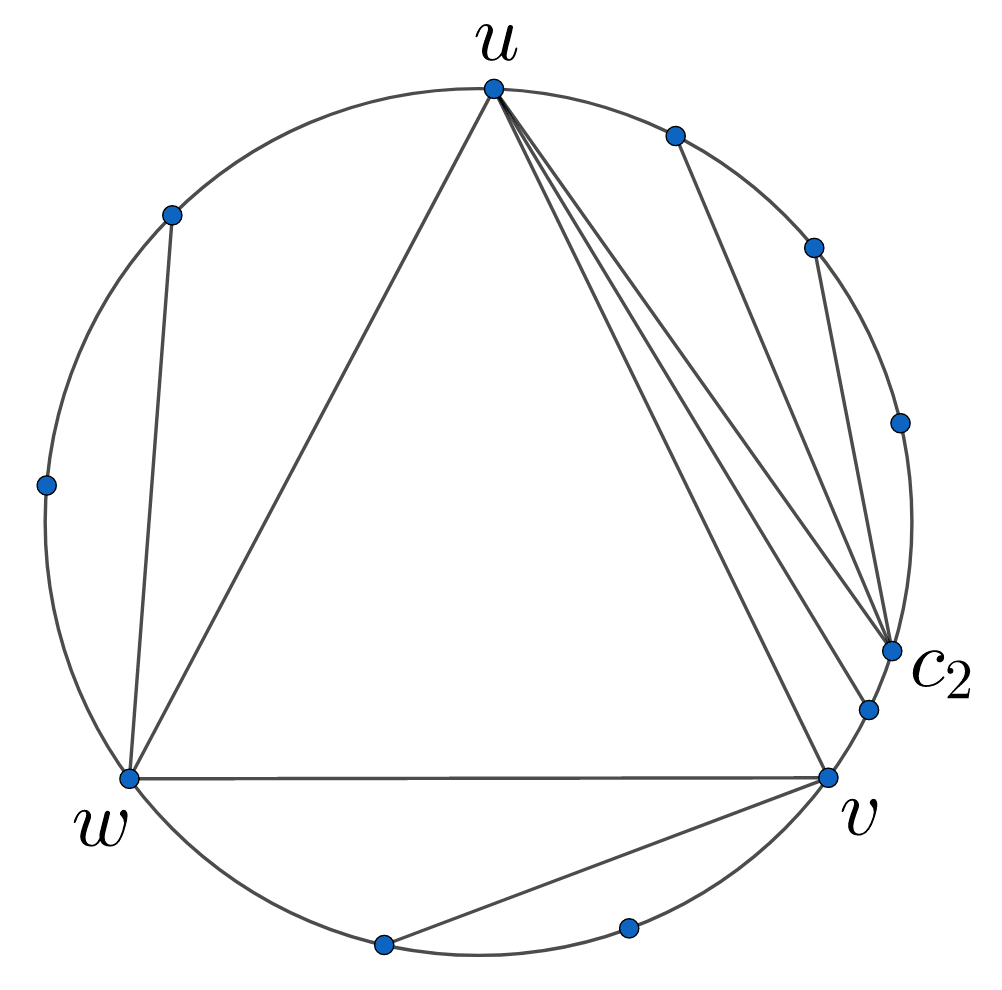}
				\hspace{0.4cm}
				\includegraphics[width=3cm,clip=false]{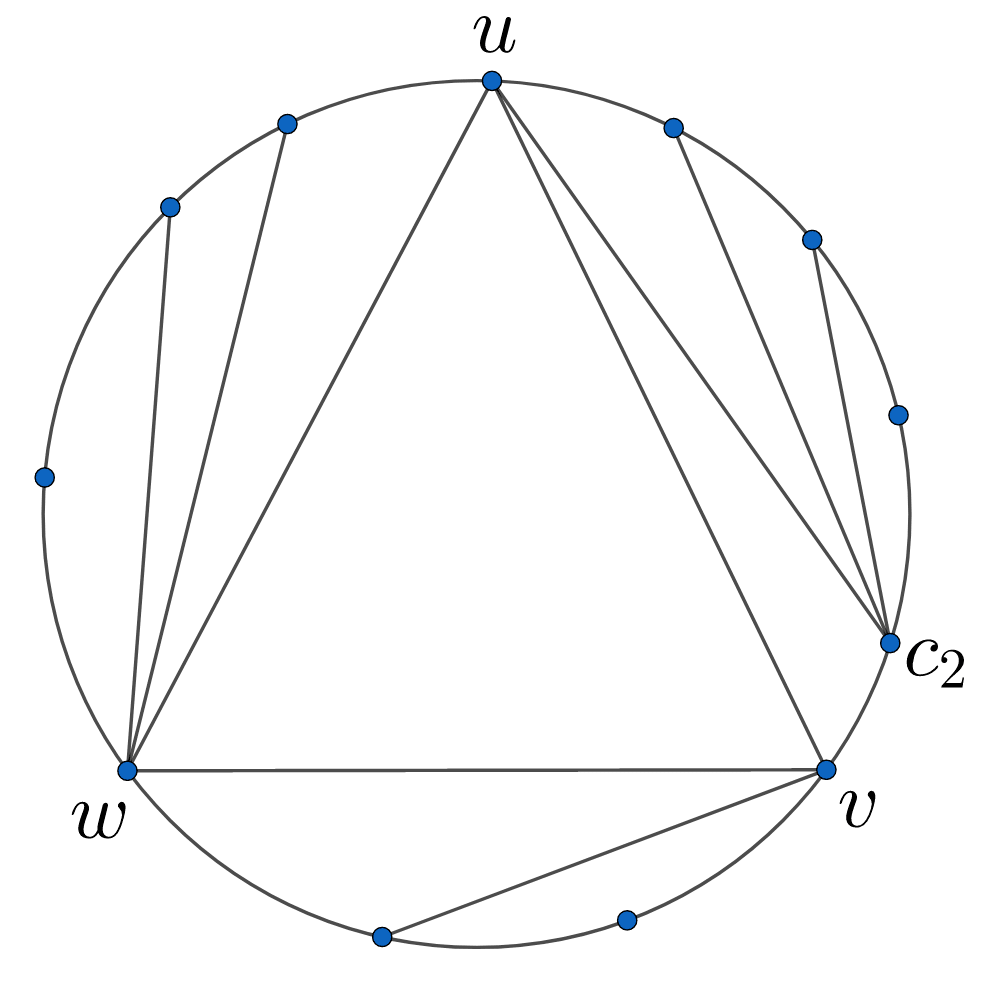}
				\caption{Here $u,c_2,v,w$ are not all of the same degree, thus $s$ is not unigraphic.}
			\end{subfigure}
			\caption{$a=3$, three non-trivial caterpillars, $\ell_1=2$, $\ell_2=\ell_3=1$.}
			\label{pic:022}
		\end{figure}
	\end{itemize}
\end{itemize}
One checks that, on the other hand, sequences of types B1, C1, D1 are all unigraphic. The proof of Theorem \ref{prop:2} is complete.

\clearpage
\bibliographystyle{abbrv}
\bibliography{bibgra}

\Addresses

\end{document}